\documentclass{article}

\usepackage[height=8in]{geometry}
\usepackage{amsmath, amssymb, mathabx, comment, mathrsfs}
\usepackage{graphicx, color}
\usepackage{array, multirow}
\usepackage[font=small,labelfont=bf]{caption} 
\usepackage{lipsum}
\usepackage{amsfonts}
\usepackage{tabularx}
\usepackage{booktabs}
\usepackage{graphicx}
\usepackage{epstopdf}
\usepackage{algorithmic}
\usepackage{mathtools}
\usepackage{tikz}
\usepackage{todonotes}
\usepackage{adjustbox}
\usepackage{float}
\usepackage{subcaption}
\definecolor{darkgreen}{rgb}{0, .5, 0}

\ifpdf
\DeclareGraphicsExtensions{.eps,.pdf,.png,.jpg}
\else
\DeclareGraphicsExtensions{.eps}
\fi

\AtBeginDocument{%
\setlength{\oddsidemargin}{\dimexpr(\paperwidth-\textwidth)/2-1in}%
\setlength{\evensidemargin}{\oddsidemargin}%
\setlength{\topmargin}{%
\dimexpr(\paperheight-\textheight)/2-\headheight-\headsep-1in}%
}

\usepackage[shortlabels]{enumitem}

\usepackage[sort]{natbib}

\setcitestyle{numbers,square,comma}
\usepackage{hyperref}
\hypersetup{colorlinks=true,
urlcolor=blue,
linkcolor=blue,
citecolor=blue,
bookmarksdepth=paragraph}
\usepackage[nameinlink]{cleveref}
\crefname{equation}{}{}
\crefname{section}{section}{sections}
\crefname{figure}{figure}{figures}
\crefname{table}{table}{tables}
\crefname{example}{example}{examples}
\crefname{proposition}{proposition}{propositions}
\Crefname{section}{Section}{Sections}
\Crefname{figure}{Figure}{Figures}
\Crefname{table}{Table}{Tables}
\Crefname{definition}{Definition}{Definitions}
\Crefname{theorem}{Theorem}{Theorems}
\Crefname{remark}{Remark}{Remarks}
\Crefname{example}{Example}{Examples}
\Crefname{proposition}{Proposition}{Propositions}
\numberwithin{equation}{section}

\usepackage{amsfonts,amsmath,amssymb,amsthm}
\newtheorem{theorem}{Theorem}[section]
\newtheorem{lemma}{Lemma}[section]

\theoremstyle{definition}
\newtheorem{definition}{Definition}[section]

\NewDocumentCommand{\E}{o}{%
    \mathbb{E}^g\IfValueT{#1}{\big( {#1} \big)}
}
    \NewDocumentCommand{\Eg}{o}{%
        \mathbb{E}^g\IfValueT{#1}{\big( {#1} \big)}
    }

\NewDocumentCommand{\Flow}{o}{%
    \operatorname{Flow}{\IfValueT{#1}{\Big({#1}\Big)}}
}
\NewDocumentCommand{\supp}{o}{%
    \operatorname{supp}{\IfValueT{#1}{\Big({#1}\Big)}}
}

\newcommand{\eqdef}{\ensuremath{\stackrel{\mbox{\upshape\tiny def.}}{=}}}

\newtheorem{set}{Setting}
\title{Structure-Preserving Reconstruction of Convex Lipschitz Functionals on Hilbert Spaces from Finite Samples}
\author{Anastasis Kratsios}
\date{ }

\usepackage{tikz,pgfplots}
\pgfplotsset{compat=1.18}
\usetikzlibrary{arrows.meta,calc}
\usepackage[T1]{fontenc}

\definecolor{pup}{rgb}{0.4980,0.4670,0.8670}
\definecolor{puplite}{rgb}{0.6860,0.6710,0.9410}
\definecolor{teal}{rgb}{0.1140,0.6200,0.4590}
\definecolor{teallite}{rgb}{0.4980,0.8040,0.6820}
\definecolor{greenc}{rgb}{0.1140,0.6200,0.4590}
\definecolor{grayax}{rgb}{0.5500,0.5500,0.5500}
\definecolor{pc0}{rgb}{0.1250,0.3750,0.8750}
\definecolor{pc1}{rgb}{0.2500,0.1250,0.8750}
\definecolor{pc2}{rgb}{0.6250,0.1250,0.8750}
\definecolor{pc3}{rgb}{0.8750,0.1250,0.7500}
\definecolor{pc4}{rgb}{0.8750,0.1250,0.3750}
\definecolor{pc5}{rgb}{0.8750,0.2500,0.1250}
\definecolor{pc6}{rgb}{0.8750,0.6250,0.1250}
\definecolor{pc7}{rgb}{0.7500,0.8750,0.1250}
\definecolor{pc8}{rgb}{0.3750,0.8750,0.1250}
\definecolor{pc9}{rgb}{0.1250,0.8750,0.2500}

\begin{document}

\maketitle

\begin{abstract}
Convex functionals are ubiquitous in applied analysis, appearing as value functions, risk measures, super-hedging prices, and loss functionals in machine learning. In many applications, however, the functional is only observed through finitely many exact pointwise evaluations. We ask whether a convex functional on a separable Hilbert space $H$ can be reconstructed, up to arbitrary uniform accuracy, by an explicit formula which preserves convexity and Lipschitz regularity and is finitely computable.

We answer this affirmatively. For every compact convex $C\subseteq H$, every $L$-Lipschitz convex functional $\rho:C\to\mathbb{R}$, and every $\varepsilon>0$, we construct an explicit finite-sample reconstruction which is convex, $L$-Lipschitz, and uniformly $\varepsilon$-accurate on $C$. The construction uses only finitely many linear measurements $\langle b,\cdot\rangle_H$, with $b$ lying in a finite-dimensional subspace of $H$, and is exactly implementable by a $\operatorname{ReLU}$-MLP. Building on this, we introduce convex neural functionals (CNFs), a structured trainable architecture class containing our reconstruction, whose every admissible parameter configuration is automatically convex and Lipschitz, providing a principled foundation for learning convex functionals from finite data.
\end{abstract}

\section{Introduction}
\label{s:Intro}
Phenomena spanning a broad range of sciences are theoretically described by convex functions which are only observed at finitely many data-points in practice.  Examples range from preferences in economics, especially in revealed-preference theory, where an actor's preferences are represented by a concave utility function and are only observed from finitely many price-consumption observations; cf.~\cite{Afriat1967,Diewert1973,Varian1982,ForgesMinelli2009}, to European option prices in finance, which are convex in their strike by option pricing theory~\cite{BreedenLitzenberger1978} but are only observed through sparse market quotes in practice~\cite{GatheralJacquier2014,bayer2025deep,buehler2026sanos}, to general shape-constrained regression, where finitely many observations are given from a convex function which one seeks to infer~\cite{Hildreth1954,GroeneboomJongbloedWellner2001,SeijoSen2011,HannahDunson2013,MazumderChoudhuryIyengarSen2019}, and to differential privacy, where convex loss, risk, and regularization functionals play a central role~\cite{dwork2008differential,zamanlooy2023strong}.

While there is a long list of such convex-function reconstruction problems, the currently available approximate reconstruction guarantees either violate convexity; e.g.\ with most deep learning universal approximation guarantees~\cite{cybenko1989approximation,hornik1991approximation,leshno1993multilayer,pinkus1999approximation,yarotsky2017error,petersen2018optimal,AlmiraLopezdeTeruelRomeroLopezVoigtlaender2021Negative,LiuBoukaiShang2022Optimal,guhring2020error,opschoor2020deep,devore2021neural,kratsios2022geometric,kratsios2022reluedge,Ismailov2023ThreeLayer,hong2024bridging,riegler2024generating,BilokopytovXanthos2026Universal,GlotzlRichters2023Helmholtz}, or they provide formulas which are not genuinely closed-form, e.g.\ involving suprema/infima which cannot be solved explicitly, or evaluations at ``infinitely parameterized'' linear functionals, such as the interpolation formulae of~\cite{azagra2017extension,azagra2017whitney,azagra2018explicit,mudarra2026sharp}, or they are finite-dimensional~\cite{BalazsGyorgySzepesvari2015,KimLee2024MaxAffine}.  
This begs the question: \textit{Does there exist an approximate \textbf{reconstruction} \textbf{formula} for convex functions which}
\begin{enumerate}
    \item[(i)] \textbf{Convexity:} is convex for every approximation error $\varepsilon>0$,
    \item[(ii)] \textbf{Universal Approximation:} can uniformly approximate any $L$-Lipschitz convex function over any compact convex subset of a separable Hilbert space $H$, for $L>0$,
    \item[(iii)] \textbf{Regularity:} is an $L$-Lipschitz function,
    \item[(iv)] \textbf{Neural Computability:} can be exactly implemented by a $\operatorname{ReLU}$-multilayer perceptron (MLP) and finitely many linear functionals of the form $\langle b,\cdot\rangle$, where $b$ belongs to a finite-dimensional subspace of $H$.
\end{enumerate}

\subsection*{Contributions.}
Our first main result, Theorem~\ref{thrm:reconstruction}, provides precisely such a \textbf{formula}; namely, one satisfying (i)-(iii).  Next, Theorem~\ref{thrm:relu_MLP_computation} guarantees that this formula is \textbf{computed} by a \textit{neural network} model; i.e.\ (iv) is satisfied by a network whose complexity is quantified explicitly and is relatively small when evaluated over low-dimensional compact convex subsets of $H$.

We provide a \textbf{certificate} in Theorem~\ref{thrm:certificate}, showing that every member of a \textbf{more structured} class of neural networks, containing our universal class and potentially more amenable to training on general tasks, is guaranteed to be convex and Lipschitz for every admissible parameter configuration.

\subsection*{Organization of Paper.}
Section~\ref{s:Prelims} contains preliminaries and background definitions required for the formulation of our main results.
Section~\ref{s:Main} presents our main results.  Section~\ref{s:Main} also includes, in \S\ref{s:Applications}, toy experiments validating our theory.  All proofs are relegated to Section~\ref{s:Proofs}.

\section{Preliminaries}
\label{s:Prelims}
We first summarize the notation used throughout our manuscript for easy reference; then, we overview the relevant deep learning models.

\subsection{Notation}
\label{s:Notation}
We use the following notation throughout.
\begin{itemize}[itemsep=2pt]
    \item $\mathbb{N}_0\eqdef\{0,1,2,\dots\}$ and $\mathbb{N}_+\eqdef\{1,2,\dots\}$.
    \item For $r\in\mathbb{N}_+$, $[r]_+\eqdef\{1,\dots,r\}$.
    \item If $H$ is a Hilbert space, then $\langle \cdot,\cdot\rangle_H$ denotes its inner product and $\|\cdot\|_H$ denotes the induced norm.  When no confusion is possible, we simply write $\langle \cdot,\cdot\rangle$ and $\|\cdot\|$.
    \item For $x\in H$ and $r>0$, $B_H(x,r)$ and $\overline{B}_H(x,r)$ denote the open and closed balls in $H$, respectively.
    \item For a non-empty set $C\subseteq H$, $\operatorname{diam}(C)\eqdef\sup_{x,y\in C}\|x-y\|_H$ denotes its diameter.
    \item For any non-empty compact set $C\subseteq H$ and any $\delta>0$, $N_\delta(C)$ denotes the smallest cardinality of a subset $\mathbb{X}\subseteq C$ such that
    $
        C\subseteq \bigcup_{\xi\in\mathbb{X}}B_H(\xi,\delta)
    $.
    \item For scalars or vectors $z_1,\dots,z_M$, $\bigoplus_{m=1}^M z_m$ denotes their concatenation.  In particular, if $z_m\in\mathbb{R}$, then $\bigoplus_{m=1}^M z_m\in\mathbb{R}^M$.
    \item For $N\in\mathbb{N}_+$, $\Delta_N$ denotes the probability simplex
    $
        \Delta_N\eqdef\{\lambda\in[0,1]^N:\sum_{n=1}^N\lambda_n=1\},
    $
    with extreme points $e_1,\dots,e_N\in\mathbb{R}^N$ denoting the standard basis vectors.
    \item For a Lipschitz map $f$, $\operatorname{Lip}(f)$ denotes its optimal Lipschitz constant.
    \item For a linear subspace $V\subseteq H$, $\operatorname{dist}(x,V)\eqdef\inf_{v\in V}\|x-v\|_H$ denotes the distance from $x\in H$ to $V$.
    \item For a bounded linear map $A:\mathbb{R}^d\to\mathbb{R}^D$, $\|A\|_{2\to 2}$ denotes its operator norm with respect to the Euclidean norm.
    \item For $M\in\mathbb{N}_+$, $\mathbf{1}_M\eqdef(1,\dots,1)\in\mathbb{R}^M$ denotes the all-ones vector.
    \item For $M\in\mathbb{N}_+$, $I_M\in\mathbb{R}^{M\times M}$ denotes the $M\times M$ identity matrix.
    \item For $d\in\mathbb{N}_+$, $\operatorname{vol}_d$ denotes the $d$-dimensional Euclidean volume (Lebesgue measure on $\mathbb{R}^d$).
    \item For an MLP, $\operatorname{depth}(\cdot)$, $\operatorname{width}(\cdot)$, and $\operatorname{size}(\cdot)$ denote respectively its number of layers, maximal hidden-layer width, and total number of affine parameters.
\end{itemize}
For completeness, we now briefly recall the relevant deep learning models.

\subsection{Deep Learning Models}
\label{s:Prelims__ss:DeepLearningModels}
We first recall the definition of the standard $\operatorname{ReLU}\eqdef \max\{\cdot,0\}$-multilayer perceptron (MLP) model.
\begin{definition}[{$\operatorname{ReLU}$-Fully Connected MLPs}]
Let $d,D,L\in\mathbb{N}_+$, and let $d_0,d_1,\dots,d_L\in\mathbb{N}_+$ be such that $d_0=d$ and $d_L=D$.
A fully-connected MLP with depth $L$ and activation function $\operatorname{ReLU}\eqdef \max\{\cdot,0\}$ is a map $\Phi:\mathbb{R}^d\to\mathbb{R}^D$ with iterative representation:
\begin{align*} 
%\label{eq:representation_MLP}
    \Phi(x)
& \eqdef  
    {\bf A}^{(L-1)} x^{(L-1)} + b^{(L-1)},
\\
    x^{(j+1)} 
& \eqdef  
    \operatorname{ReLU}^{\bullet}({\bf A}^{(j)} x^{(j)} + b^{(j)}) 
    \qquad 
    \mbox{for } 
    j=0,\dots,L-2,
\\
    x^{(0)} 
& \eqdef  
    x, 
\end{align*}
where $\operatorname{ReLU}^{\bullet}$ denotes component-wise application, ${\bf A}^{(j)}\in\mathbb{R}^{d_{j+1}\times d_j}$ and $b^{(j)}\in\mathbb{R}^{d_{j+1}}$ for $j=0,\dots,L-1$.
Here, $L$ is the depth of the MLP and $\Upsilon\eqdef \max_{j=1,\dots,L-1}\,d_j$ is its width.
\end{definition}
This is in the same spirit as input-convex neural networks
\cite{amos2017input} and monotone networks \cite{sill1997monotonic};
see also \cite{liu2020certified} for certification of monotonicity in
general piecewise-linear neural networks.  Our certifiable results concern the following class of feedforward neural networks with infinite-dimensional inputs; thus forming a class of neural operators, cf.~\cite{lu2021learning,li2021fourier,kovachki2023neural,li2020multipole,li2024physics,kratsios2024mixture,kratsios2022universal,papon2022universal}, perhaps more accurately called \textit{neural functionals} since their codomain is the real line $\mathbb{R}$.  This class also incorporates \textit{max-pooling} layers of~\cite{YamaguchiSakamotoAkabaneFujimoto1990MaxPooling}, popularized in convolutional neural networks (CNNs) by~\cite{WengAhujaHuang1993Cresceptron}.  We remark that incorporating such layers does not impede the network's capacity to generalize~\cite{KratsiosCousinsSaezdeOcarizBordeKimBrugiapaglia2026}; although we do not explore statistical properties here, and focus instead on analytic and approximation-theoretic expressivity/representation-capacity questions.

The gated, or parametric, version of the $\operatorname{ReLU}$ function is the
\textit{parametric ReLU} function, denoted by $\operatorname{PReLU}$.  It acts on
any $u\in\mathbb{R}^n$, for $n\in\mathbb{N}_+$, by
\[
        \operatorname{PReLU}^{\bullet}_{\mathbf{\alpha}}(u)
    =
        \big(
            \operatorname{ReLU}(u_i) - \mathbf{\alpha}_i\,
            \operatorname{ReLU}(-u_i)
        \big)_{i=1}^n,
\]
where $\mathbf{\alpha}\in [0,1]^n$ is the gating/slope parameter.

\begin{definition}[Convex Neural Functionals (CNF)]
\label{defn:CNFs}
Let $H$ be a real Hilbert space, let $d,M,L\in\mathbb{N}_+$ with $L\ge 2$ and $\dim(H)\ge d$, and let
$V\subseteq H$ be a $d$-dimensional linear subspace.  Let
$p_1,\dots,p_M\in V$ and let $q_1,\dots,q_M\in\mathbb{R}$.
Let $d_0,d_1,\dots,d_L\in\mathbb{N}_+$ and
$d_1^{\prime},\dots,d_{L-1}^{\prime}\in\mathbb{N}_+$ be such that
$d_0=M$, $d_L=1$, and $d_\ell^{\prime}\ge d_\ell$ for each $\ell=1,\dots,L-1$.
A map $\Phi:H\to\mathbb{R}$ is called a \textbf{convex neural functional} if it admits the following iterative representation:
\begin{align}
\label{eq:CNFs}
    \Phi(x)
    & \eqdef  
    {\bf A}^{(L-1)} x^{(L-1)} + b^{(L-1)}
\\
\label{eq:maxpooling}
    x^{(\ell+1)} 
    & \eqdef  
    \max_{\Pi^{(\ell)}}
    \operatorname{PReLU}^{\bullet}_{\mathbf{\alpha}^{(\ell)}}({\bf A}^{(\ell)} x^{(\ell)} + b^{(\ell)}) 
    \qquad 
    \mbox{ for } 
    \qquad
    \ell=0,1,\dots,L-2,
\\
    x^{(0)} 
    & \eqdef    
        \left(
        \bigoplus_{m=1}^M
        \left[
            \langle p_m,x\rangle_H
            +
            q_m
        \right]
    \right)
\end{align}
where 
$
    {\bf A}^{(\ell)}\in [0,\infty)^{d_{\ell+1}^{\prime}\times d_\ell}
$, $
    b^{(\ell)}\in\mathbb{R}^{d_{\ell+1}^{\prime}}
$ and $
\mathbf{\alpha}^{(\ell)}\in [0,1]^{d_{\ell+1}^{\prime}}
$ for each  $\ell=0,\dots,L-2$, 
and $
    {\bf A}^{(L-1)}\in [0,\infty)^{1\times d_{L-1}}
$ for  $b^{(L-1)}\in\mathbb{R}$. 
For each $\ell=0,\dots,L-2$, $\Pi^{(\ell)}=\{[[k]]_{\Pi^{(\ell)}}\}_{k=1}^{d_{\ell+1}}$
is a partition of $[d_{\ell+1}^{\prime}]_+$ into $d_{\ell+1}$ non-empty parts, and for each $k\in [d_{\ell+1}]_+$ we define the \textit{max-pooling} layer
\begin{equation}
\label{eq:max_pooling}
\begin{aligned}
    \max_{\Pi^{(\ell)}}:\mathbb{R}^{d_{\ell+1}^{\prime}} 
    & \rightarrow 
    \mathbb{R}^{d_{\ell+1}},
\\
    \max_{\Pi^{(\ell)}}(z)_k 
    & \eqdef 
    \max_{i\in [[k]]_{\Pi^{(\ell)}}}\, z_i
.
\end{aligned}
\end{equation}
Here, $L$ is the \textbf{depth} of the MLP,
$\Upsilon\eqdef \max_{\ell=1,\dots,L-1}\,d_\ell$ is its \textbf{width}, and $d$ is the \textbf{rank} of $\Phi$.
\end{definition}

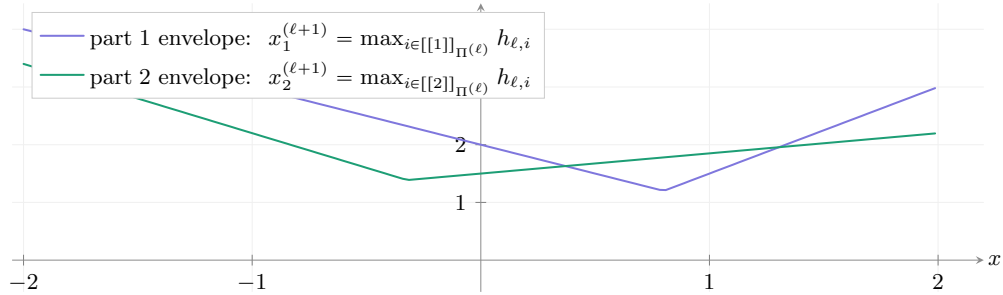
\begin{figure}[hb!]
\centering
\begin{tikzpicture}[scale=0.95, transform shape]
\begin{axis}[
  width=\textwidth, height=5.6cm,
  xmin=-2.05, xmax=2.2, ymin=-0.55, ymax=4.45,
  axis lines=center,
  xtick={-2,-1,0,1,2}, ytick={0,1,2,3,4},
  tick label style={font=\small},
  xlabel={$x$},
  xlabel style={at={(axis cs:2.18,0)}, anchor=west, font=\small},
  axis line style={grayax, thin},
  grid=both,
  grid style={gray!12, thin},
  clip=true,
  legend style={at={(0.02,0.97)}, anchor=north west,
                font=\small, draw=gray!40, fill=white,
                row sep=1pt},
]
 
\addplot[pup, thick, line cap=round] coordinates {(-2.000,4.0000) (-1.973,3.9732) (-1.946,3.9465) (-1.920,3.9197) (-1.893,3.8930) (-1.866,3.8662) (-1.839,3.8395) (-1.813,3.8127) (-1.786,3.7860) (-1.759,3.7592) (-1.732,3.7324) (-1.706,3.7057) (-1.679,3.6789) (-1.652,3.6522) (-1.625,3.6254) (-1.599,3.5987) (-1.572,3.5719) (-1.545,3.5452) (-1.518,3.5184) (-1.492,3.4916) (-1.465,3.4649) (-1.438,3.4381) (-1.411,3.4114) (-1.385,3.3846) (-1.358,3.3579) (-1.331,3.3311) (-1.304,3.3043) (-1.278,3.2776) (-1.251,3.2508) (-1.224,3.2241) (-1.197,3.1973) (-1.171,3.1706) (-1.144,3.1438) (-1.117,3.1171) (-1.090,3.0903) (-1.064,3.0635) (-1.037,3.0368) (-1.010,3.0100) (-0.983,2.9833) (-0.957,2.9565) (-0.930,2.9298) (-0.903,2.9030) (-0.876,2.8763) (-0.849,2.8495) (-0.823,2.8227) (-0.796,2.7960) (-0.769,2.7692) (-0.742,2.7425) (-0.716,2.7157) (-0.689,2.6890) (-0.662,2.6622) (-0.635,2.6355) (-0.609,2.6087) (-0.582,2.5819) (-0.555,2.5552) (-0.528,2.5284) (-0.502,2.5017) (-0.475,2.4749) (-0.448,2.4482) (-0.421,2.4214) (-0.395,2.3946) (-0.368,2.3679) (-0.341,2.3411) (-0.314,2.3144) (-0.288,2.2876) (-0.261,2.2609) (-0.234,2.2341) (-0.207,2.2074) (-0.181,2.1806) (-0.154,2.1538) (-0.127,2.1271) (-0.100,2.1003) (-0.074,2.0736) (-0.047,2.0468) (-0.020,2.0201) (0.007,1.9933) (0.033,1.9666) (0.060,1.9398) (0.087,1.9130) (0.114,1.8863) (0.140,1.8595) (0.167,1.8328) (0.194,1.8060) (0.221,1.7793) (0.247,1.7525) (0.274,1.7258) (0.301,1.6990) (0.328,1.6722) (0.355,1.6455) (0.381,1.6187) (0.408,1.5920) (0.435,1.5652) (0.462,1.5385) (0.488,1.5117) (0.515,1.4849) (0.542,1.4582) (0.569,1.4314) (0.595,1.4047) (0.622,1.3779) (0.649,1.3512) (0.676,1.3244) (0.702,1.2977) (0.729,1.2709) (0.756,1.2441) (0.783,1.2174) (0.809,1.2140) (0.836,1.2542) (0.863,1.2943) (0.890,1.3344) (0.916,1.3746) (0.943,1.4147) (0.970,1.4548) (0.997,1.4950) (1.023,1.5351) (1.050,1.5753) (1.077,1.6154) (1.104,1.6555) (1.130,1.6957) (1.157,1.7358) (1.184,1.7759) (1.211,1.8161) (1.237,1.8562) (1.264,1.8963) (1.291,1.9365) (1.318,1.9766) (1.344,2.0167) (1.371,2.0569) (1.398,2.0970) (1.425,2.1371) (1.452,2.1773) (1.478,2.2174) (1.505,2.2575) (1.532,2.2977) (1.559,2.3378) (1.585,2.3779) (1.612,2.4181) (1.639,2.4582) (1.666,2.4983) (1.692,2.5385) (1.719,2.5786) (1.746,2.6187) (1.773,2.6589) (1.799,2.6990) (1.826,2.7391) (1.853,2.7793) (1.880,2.8194) (1.906,2.8595) (1.933,2.8997) (1.960,2.9398) (1.987,2.9799)};
\addlegendentry{part\;$1$ envelope: \;$x^{(\ell+1)}_1 = \max_{i\in[[1]]_{\Pi^{(\ell)}}} h_{\ell,i}$}
\addplot[teal, thick, line cap=round] coordinates {(-2.000,3.4000) (-1.973,3.3679) (-1.946,3.3358) (-1.920,3.3037) (-1.893,3.2716) (-1.866,3.2395) (-1.839,3.2074) (-1.813,3.1753) (-1.786,3.1431) (-1.759,3.1110) (-1.732,3.0789) (-1.706,3.0468) (-1.679,3.0147) (-1.652,2.9826) (-1.625,2.9505) (-1.599,2.9184) (-1.572,2.8863) (-1.545,2.8542) (-1.518,2.8221) (-1.492,2.7900) (-1.465,2.7579) (-1.438,2.7258) (-1.411,2.6936) (-1.385,2.6615) (-1.358,2.6294) (-1.331,2.5973) (-1.304,2.5652) (-1.278,2.5331) (-1.251,2.5010) (-1.224,2.4689) (-1.197,2.4368) (-1.171,2.4047) (-1.144,2.3726) (-1.117,2.3405) (-1.090,2.3084) (-1.064,2.2763) (-1.037,2.2441) (-1.010,2.2120) (-0.983,2.1799) (-0.957,2.1478) (-0.930,2.1157) (-0.903,2.0836) (-0.876,2.0515) (-0.849,2.0194) (-0.823,1.9873) (-0.796,1.9552) (-0.769,1.9231) (-0.742,1.8910) (-0.716,1.8589) (-0.689,1.8268) (-0.662,1.7946) (-0.635,1.7625) (-0.609,1.7304) (-0.582,1.6983) (-0.555,1.6662) (-0.528,1.6341) (-0.502,1.6020) (-0.475,1.5699) (-0.448,1.5378) (-0.421,1.5057) (-0.395,1.4736) (-0.368,1.4415) (-0.341,1.4094) (-0.314,1.3900) (-0.288,1.3993) (-0.261,1.4087) (-0.234,1.4181) (-0.207,1.4274) (-0.181,1.4368) (-0.154,1.4462) (-0.127,1.4555) (-0.100,1.4649) (-0.074,1.4742) (-0.047,1.4836) (-0.020,1.4930) (0.007,1.5023) (0.033,1.5117) (0.060,1.5211) (0.087,1.5304) (0.114,1.5398) (0.140,1.5492) (0.167,1.5585) (0.194,1.5679) (0.221,1.5773) (0.247,1.5866) (0.274,1.5960) (0.301,1.6054) (0.328,1.6147) (0.355,1.6241) (0.381,1.6334) (0.408,1.6428) (0.435,1.6522) (0.462,1.6615) (0.488,1.6709) (0.515,1.6803) (0.542,1.6896) (0.569,1.6990) (0.595,1.7084) (0.622,1.7177) (0.649,1.7271) (0.676,1.7365) (0.702,1.7458) (0.729,1.7552) (0.756,1.7645) (0.783,1.7739) (0.809,1.7833) (0.836,1.7926) (0.863,1.8020) (0.890,1.8114) (0.916,1.8207) (0.943,1.8301) (0.970,1.8395) (0.997,1.8488) (1.023,1.8582) (1.050,1.8676) (1.077,1.8769) (1.104,1.8863) (1.130,1.8957) (1.157,1.9050) (1.184,1.9144) (1.211,1.9237) (1.237,1.9331) (1.264,1.9425) (1.291,1.9518) (1.318,1.9612) (1.344,1.9706) (1.371,1.9799) (1.398,1.9893) (1.425,1.9987) (1.452,2.0080) (1.478,2.0174) (1.505,2.0268) (1.532,2.0361) (1.559,2.0455) (1.585,2.0548) (1.612,2.0642) (1.639,2.0736) (1.666,2.0829) (1.692,2.0923) (1.719,2.1017) (1.746,2.1110) (1.773,2.1204) (1.799,2.1298) (1.826,2.1391) (1.853,2.1485) (1.880,2.1579) (1.906,2.1672) (1.933,2.1766) (1.960,2.1860) (1.987,2.1953)};
\addlegendentry{part\;$2$ envelope: \;$x^{(\ell+1)}_2 = \max_{i\in[[2]]_{\Pi^{(\ell)}}} h_{\ell,i}$}
\end{axis}
\end{tikzpicture}
\caption{Max-pooling (cf.~\eqref{eq:maxpooling}) as upper envelope.
Each partition group collects several
$\operatorname{PReLU}$-activated pre-pooling neurons $h_{\ell,i}$
and returns their pointwise maximum $x^{(\ell+1)}_k = \max_{i\in[[k]]_{\Pi^{(\ell)}}}h_{\ell,i}$.
Since the pointwise maximum of convex functions is convex,
every output neuron of the max-pooling layer is automatically convex.}
\end{figure}

\section{Main Results}
\label{s:Main}
We operate in the following setting.
\begin{set}
\label{setting}
Fix $L>0$.  Let $H$ be a nonzero separable Hilbert space, let $C\subseteq H$ be compact and convex with
$D\eqdef \operatorname{diam}(C)>0$, and let $\rho:C\to\mathbb{R}$ be a convex and $L$-Lipschitz map.
\end{set}

\begin{theorem}[Convexity- and Regularity-Preserving: Reconstruction Formula from Finite Data]
\label{thrm:reconstruction}
In Setting~\ref{setting}, fix any $\varepsilon>0$, and let
$\mathbb{X}_N\eqdef\{\xi_n\}_{n=1}^N\subseteq C$ be a $\mathcal{O}(\varepsilon)$-net of $C$ with
$N=N_{\mathcal{O}(\varepsilon)}(C)$\footnote{Precise constants are given in Theorem~\ref{thrm:reconstruction__full}.}.
\hfill\\
Then, there exist $d\in\mathbb{N}_0$, a $d$-dimensional linear subspace
$V_d\subseteq H$, and an $\mathcal{O}(\varepsilon)$-net
$p \eqdef\{p_m\}_{m=1}^M$ of $V_d\cap \overline{B}_H(0,L)$ such that the induced map
$f_{\varepsilon,d}:H\to \mathbb{R}$ defined for each $x\in H$ by
\begin{equation}
\label{eq:I_dualized_discretized__hellooo}
    f_{\varepsilon,d}(x)
\eqdef
    \max_{m=1,\dots,M}
    \,
    \Big\{
        \langle p_m,x\rangle
        +
        \min_{1\leq n\leq N}
        \bigl(
            \rho(\xi_n)-\langle p_m,\xi_n\rangle
        \bigr)
    \Big\}
\end{equation}
is convex, $L$-Lipschitz, and satisfies 
\begin{equation}
\label{eq:das_uniformli_estimatoski_samples}
    \sup_{\xi\in C}
    \,
        \bigl|
            \rho(\xi)
            -
            f_{\varepsilon,d}(\xi)
        \bigr|
    <
        \varepsilon.
\end{equation}
Moreover,
$
    N
\le 
    N_{\mathcal{O}(\varepsilon)}(C)
$, 
$
    d
\leq
    N_{\mathcal{O}(\varepsilon)}(C)
$, 
and $
    M
\leq
    \mathcal{O}(\varepsilon^{-d})
$.
\end{theorem}

The reconstruction formula in~\eqref{eq:I_dualized_discretized__hellooo} is only genuinely useful if it can be ``computed''; by this, we mean computed by a neural network model built only with $\operatorname{ReLU}$-MLPs and allowed access to inner-products in $H$ against finite-dimensional, hence realistically representable, vectors therein.  Moreover, following the breakthroughs made by \textit{in-context learning}, we also allow the map to be built using the training ``\textbf{context} itself,'' i.e.\ the input-output pairs of \textit{training data}.

\begin{theorem}[Neural network representability: finite-sample version]
\label{thrm:relu_MLP_computation}
In setting~\ref{setting}.  
Fix arbitrary paired data
$\{(\xi_n,y_n)\}_{n=1}^N\subseteq C\times\mathbb{R}$ with $y_n=\rho(\xi_n)$.
For every $\varepsilon>0$, 
there exist $M\in\mathbb{N}_+$, points
$\{p_m\}_{m=1}^M\subseteq \overline B_H(0,L)$, a \textit{convex} $\operatorname{ReLU}$-MLP
$\Phi_M:\mathbb{R}^M\to\mathbb{R}$, and a \textit{concave} $\operatorname{ReLU}$-MLP
$\Psi_N:\mathbb{R}^N\to\mathbb{R}$ such that the neural-network  $f_{\varepsilon,d}:H\to\mathbb{R}$ defined by
\begin{equation}
\label{eq:relu_representability_formula}
    f_{\varepsilon,d}(x)
    \eqdef
    \Phi_M
    \Big(
        \bigoplus_{m=1}^M
        \left[
            \langle p_m,x\rangle
            +
            \Psi_N
            \Big(
                \bigoplus_{n=1}^N
                \bigl(
                    y_n-\langle p_m,\xi_n\rangle
                \bigr)
            \Big)
        \right]
    \Big)
\end{equation}
satisfies 
\[
        \max_{1\leq n\leq N}
        \,
        \bigl|
            \rho(\xi_n)-f_{\varepsilon,d}(\xi_n)
        \bigr|
    <
        \varepsilon
\]
Moreover, for $M,N\geq 2$, these networks may be chosen so that
$
    \operatorname{size}(\Phi_M)\leq 16M
$, $\operatorname{width}(\Phi_M)\leq 3M$, $
    \operatorname{depth}(\Phi_M)\leq \lceil\log_2(M)\rceil
$, $
    \operatorname{size}(\Psi_N)\leq 16N
$, 
   $
    \operatorname{width}(\Psi_N)\leq 3N
$, and $
    \operatorname{depth}(\Psi_N)\leq \lceil\log_2(N)\rceil
$.  
\hfill\\
In particular, if $\{\xi_n\}_{n=1}^N$ contains an $\varepsilon/(4L)$-net of $C$, then~\eqref{eq:das_uniformli_estimatoski_samples} holds.
\end{theorem}

% \begin{figure}[H]%[htb!]
%     \centering
%     \includegraphics[width=0.95\linewidth]{CNF_Explanation.pdf}
%     \caption{\textbf{Cartoon intuition behind Formula~\ref{eq:relu_representability_formula}:}
%     In the case $d=2$, the left panel shows the finite family of directions
%     $p_m$ used to discretize the dual ball.  Each direction determines one affine
%     piece
%     $
%     x\mapsto
%     \langle p_m,x\rangle+\min_n(\rho(\xi_n)-\langle p_m,\xi_n\rangle),
%     $
%     whose slope is $p_m$ and whose offset is selected from the samples.  The right
%     panel restricts these affine pieces to the line $x=(t,0)$, $t\in[0,1]$.  The
%     faint blue lines are the individual affine pieces, while the green curve is
%     their pointwise maximum, namely $f_{\varepsilon,d}$.  Thus convexity is visible
%     as an upper envelope of affine functions, while Lipschitzness is encoded by
%     the uniform slope bound $\|p_m\|\le L$.  The samples
%     $(\xi_n,\rho(\xi_n))$ are shown only on the green graph, with their colours
%     indicating the directions selecting the corresponding active affine pieces.}
%     \label{fig:the_model}
% \end{figure}

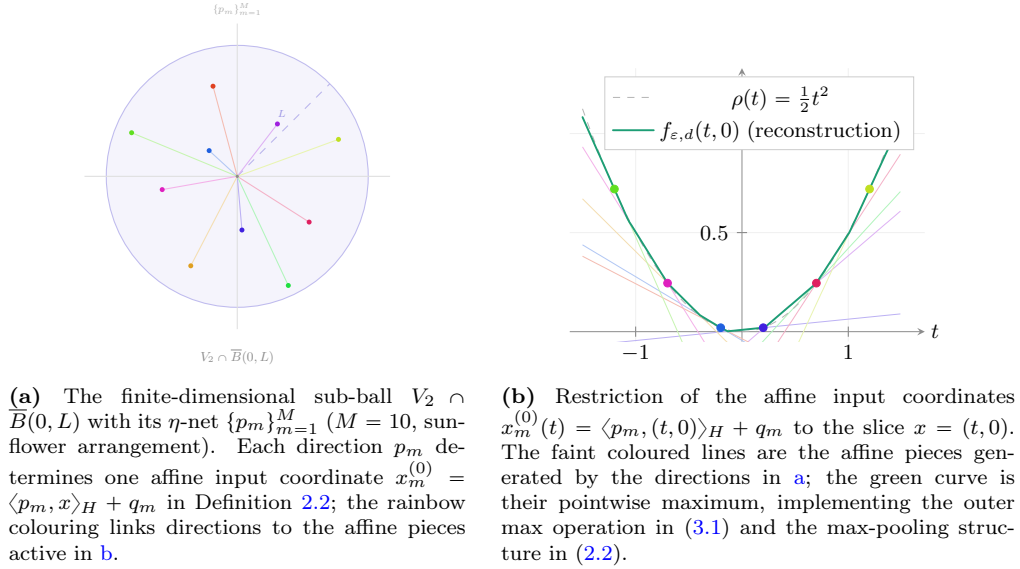
\begin{figure}[t]
\centering
 
\begin{subfigure}[b]{0.4\textwidth}
\centering
\begin{tikzpicture}[scale=0.55, transform shape]
\fill[pup!8] (0,0) circle (3.1500000000000004cm);
\draw[pup!50, thin] (0,0) circle (3.1500000000000004cm);
\draw[gray!25,thin] (-3.675cm,0)--(3.675cm,0);
\draw[gray!25,thin] (0,-3.675cm)--(0,3.675cm);
\draw[pup!55, dashed, thin] (0,0)--(2.227cm,2.227cm) node[pos=0.58,above left,font=\scriptsize,pup!80]{$L$};
\draw[pc0!35,thin] (0,0)--(-0.676cm,0.619cm);
\draw[pc1!35,thin] (0,0)--(0.113cm,-1.291cm);
\draw[pc2!35,thin] (0,0)--(0.966cm,1.260cm);
\draw[pc3!35,thin] (0,0)--(-1.805cm,-0.319cm);
\draw[pc4!35,thin] (0,0)--(1.729cm,-1.100cm);
\draw[pc5!35,thin] (0,0)--(-0.583cm,2.168cm);
\draw[pc6!35,thin] (0,0)--(-1.118cm,-2.152cm);
\draw[pc7!35,thin] (0,0)--(2.435cm,0.889cm);
\draw[pc8!35,thin] (0,0)--(-2.541cm,1.049cm);
\draw[pc9!35,thin] (0,0)--(1.228cm,-2.625cm);
\fill[pc0] (-0.676cm,0.619cm) circle (1.7pt);
\fill[pc1] (0.113cm,-1.291cm) circle (1.7pt);
\fill[pc2] (0.966cm,1.260cm) circle (1.7pt);
\fill[pc3] (-1.805cm,-0.319cm) circle (1.7pt);
\fill[pc4] (1.729cm,-1.100cm) circle (1.7pt);
\fill[pc5] (-0.583cm,2.168cm) circle (1.7pt);
\fill[pc6] (-1.118cm,-2.152cm) circle (1.7pt);
\fill[pc7] (2.435cm,0.889cm) circle (1.7pt);
\fill[pc8] (-2.541cm,1.049cm) circle (1.7pt);
\fill[pc9] (1.228cm,-2.625cm) circle (1.7pt);
\fill[grayax] (0,0) circle (1.1pt);
\node[font=\small,grayax,below] at (0,{-4.032cm}){$V_2\cap\overline{B}(0,L)$};
\node[font=\scriptsize,grayax!70,above] at (0,{3.738cm}){\strut$\{p_m\}_{m=1}^{M}$};
\end{tikzpicture}
\caption{The finite-dimensional sub-ball
$V_2\cap\overline{B}(0,L)$ with its $\eta$-net
$\{p_m\}_{m=1}^{M}$ ($M=10$, sunflower arrangement).
Each direction $p_m$ determines one affine input coordinate
$x^{(0)}_m=\langle p_m,x\rangle_H+q_m$ in Definition~\ref{defn:CNFs};
the rainbow colouring links directions to the affine pieces active
in~\subref{fig:recon}.}
\label{fig:disk}
\end{subfigure}
\quad
\begin{subfigure}[b]{0.45\textwidth}
\centering
\begin{tikzpicture}[scale=0.9, transform shape]
\begin{axis}[
  width=\textwidth, height=5.6cm,
  xmin=-1.62, xmax=1.72, ymin=-0.05, ymax=1.33,
  axis lines=center,
  xtick={-1,0,1}, ytick={0,0.5,1},
  tick label style={font=\small},
  xlabel={$t$},
  xlabel style={at={(axis cs:1.70,0)}, anchor=west, font=\small},
  axis line style={grayax, thin},
  grid=both,
  grid style={gray!12, thin},
  clip=true,
  legend style={at={(0.97,0.97)}, anchor=north east,
                font=\small, draw=gray!40, fill=white,
                row sep=1pt},
]
 
\addplot[pc0!35, thin, forget plot] coordinates {(-1.500,0.4383) (-1.480,0.4319) (-1.460,0.4254) (-1.440,0.4189) (-1.420,0.4125) (-1.400,0.4060) (-1.380,0.3996) (-1.360,0.3931) (-1.339,0.3867) (-1.319,0.3802) (-1.299,0.3737) (-1.279,0.3673) (-1.259,0.3608) (-1.239,0.3544) (-1.219,0.3479) (-1.199,0.3415) (-1.179,0.3350) (-1.159,0.3285) (-1.139,0.3221) (-1.119,0.3156) (-1.099,0.3092) (-1.079,0.3027) (-1.059,0.2963) (-1.038,0.2898) (-1.018,0.2833) (-0.998,0.2769) (-0.978,0.2704) (-0.958,0.2640) (-0.938,0.2575) (-0.918,0.2511) (-0.898,0.2446) (-0.878,0.2381) (-0.858,0.2317) (-0.838,0.2252) (-0.818,0.2188) (-0.798,0.2123) (-0.778,0.2059) (-0.758,0.1994) (-0.737,0.1929) (-0.717,0.1865) (-0.697,0.1800) (-0.677,0.1736) (-0.657,0.1671) (-0.637,0.1607) (-0.617,0.1542) (-0.597,0.1477) (-0.577,0.1413) (-0.557,0.1348) (-0.537,0.1284) (-0.517,0.1219) (-0.497,0.1155) (-0.477,0.1090) (-0.457,0.1025) (-0.436,0.0961) (-0.416,0.0896) (-0.396,0.0832) (-0.376,0.0767) (-0.356,0.0703) (-0.336,0.0638) (-0.316,0.0573) (-0.296,0.0509) (-0.276,0.0444) (-0.256,0.0380) (-0.236,0.0315) (-0.216,0.0251) (-0.196,0.0186) (-0.176,0.0121) (-0.156,0.0057) (-0.135,-0.0008) (-0.115,-0.0072) (-0.095,-0.0137) (-0.075,-0.0201) (-0.055,-0.0266) (-0.035,-0.0331) (-0.015,-0.0395) (0.005,-0.0460) (0.025,-0.0524) (0.045,-0.0589) (0.065,-0.0653) (0.085,-0.0718) (0.105,-0.0783) (0.125,-0.0847) (0.145,-0.0912) (0.166,-0.0976) (0.186,-0.1041) (0.206,-0.1105) (0.226,-0.1170) (0.246,-0.1235) (0.266,-0.1299) (0.286,-0.1364) (0.306,-0.1428) (0.326,-0.1493) (0.346,-0.1557) (0.366,-0.1622) (0.386,-0.1687) (0.406,-0.1751) (0.426,-0.1816) (0.446,-0.1880) (0.467,-0.1945) (0.487,-0.2009) (0.507,-0.2074) (0.527,-0.2139) (0.547,-0.2203) (0.567,-0.2268) (0.587,-0.2332) (0.607,-0.2397) (0.627,-0.2461) (0.647,-0.2526) (0.667,-0.2591) (0.687,-0.2655) (0.707,-0.2720) (0.727,-0.2784) (0.747,-0.2849) (0.768,-0.2913) (0.788,-0.2978)};
\addplot[pc1!35, thin, forget plot] coordinates {(-1.500,-0.0717) (-1.480,-0.0706) (-1.460,-0.0696) (-1.440,-0.0685) (-1.420,-0.0674) (-1.400,-0.0663) (-1.380,-0.0652) (-1.360,-0.0641) (-1.339,-0.0631) (-1.319,-0.0620) (-1.299,-0.0609) (-1.279,-0.0598) (-1.259,-0.0587) (-1.239,-0.0576) (-1.219,-0.0566) (-1.199,-0.0555) (-1.179,-0.0544) (-1.159,-0.0533) (-1.139,-0.0522) (-1.119,-0.0512) (-1.099,-0.0501) (-1.079,-0.0490) (-1.059,-0.0479) (-1.038,-0.0468) (-1.018,-0.0457) (-0.998,-0.0447) (-0.978,-0.0436) (-0.958,-0.0425) (-0.938,-0.0414) (-0.918,-0.0403) (-0.898,-0.0392) (-0.878,-0.0382) (-0.858,-0.0371) (-0.838,-0.0360) (-0.818,-0.0349) (-0.798,-0.0338) (-0.778,-0.0327) (-0.758,-0.0317) (-0.737,-0.0306) (-0.717,-0.0295) (-0.697,-0.0284) (-0.677,-0.0273) (-0.657,-0.0262) (-0.637,-0.0252) (-0.617,-0.0241) (-0.597,-0.0230) (-0.577,-0.0219) (-0.557,-0.0208) (-0.537,-0.0198) (-0.517,-0.0187) (-0.497,-0.0176) (-0.477,-0.0165) (-0.457,-0.0154) (-0.436,-0.0143) (-0.416,-0.0133) (-0.396,-0.0122) (-0.376,-0.0111) (-0.356,-0.0100) (-0.336,-0.0089) (-0.316,-0.0078) (-0.296,-0.0068) (-0.276,-0.0057) (-0.256,-0.0046) (-0.236,-0.0035) (-0.216,-0.0024) (-0.196,-0.0013) (-0.176,-0.0003) (-0.156,0.0008) (-0.135,0.0019) (-0.115,0.0030) (-0.095,0.0041) (-0.075,0.0051) (-0.055,0.0062) (-0.035,0.0073) (-0.015,0.0084) (0.005,0.0095) (0.025,0.0106) (0.045,0.0116) (0.065,0.0127) (0.085,0.0138) (0.105,0.0149) (0.125,0.0160) (0.145,0.0171) (0.166,0.0181) (0.186,0.0192) (0.206,0.0203) (0.226,0.0214) (0.246,0.0225) (0.266,0.0236) (0.286,0.0246) (0.306,0.0257) (0.326,0.0268) (0.346,0.0279) (0.366,0.0290) (0.386,0.0301) (0.406,0.0311) (0.426,0.0322) (0.446,0.0333) (0.467,0.0344) (0.487,0.0355) (0.507,0.0365) (0.527,0.0376) (0.547,0.0387) (0.567,0.0398) (0.587,0.0409) (0.607,0.0420) (0.627,0.0430) (0.647,0.0441) (0.667,0.0452) (0.687,0.0463) (0.707,0.0474) (0.727,0.0485) (0.747,0.0495) (0.768,0.0506) (0.788,0.0517) (0.808,0.0528) (0.828,0.0539) (0.848,0.0550) (0.868,0.0560) (0.888,0.0571) (0.908,0.0582) (0.928,0.0593) (0.948,0.0604) (0.968,0.0614) (0.988,0.0625) (1.008,0.0636) (1.028,0.0647) (1.048,0.0658) (1.069,0.0669) (1.089,0.0679) (1.109,0.0690) (1.129,0.0701) (1.149,0.0712) (1.169,0.0723) (1.189,0.0734) (1.209,0.0744) (1.229,0.0755) (1.249,0.0766) (1.269,0.0777) (1.289,0.0788) (1.309,0.0799) (1.329,0.0809) (1.349,0.0820) (1.370,0.0831) (1.390,0.0842) (1.410,0.0853) (1.430,0.0864) (1.450,0.0874) (1.470,0.0885) (1.490,0.0896)};
\addplot[pc2!35, thin, forget plot] coordinates {(-0.477,-0.2961) (-0.457,-0.2869) (-0.436,-0.2776) (-0.416,-0.2684) (-0.396,-0.2592) (-0.376,-0.2500) (-0.356,-0.2407) (-0.336,-0.2315) (-0.316,-0.2223) (-0.296,-0.2130) (-0.276,-0.2038) (-0.256,-0.1946) (-0.236,-0.1854) (-0.216,-0.1761) (-0.196,-0.1669) (-0.176,-0.1577) (-0.156,-0.1484) (-0.135,-0.1392) (-0.115,-0.1300) (-0.095,-0.1208) (-0.075,-0.1115) (-0.055,-0.1023) (-0.035,-0.0931) (-0.015,-0.0838) (0.005,-0.0746) (0.025,-0.0654) (0.045,-0.0562) (0.065,-0.0469) (0.085,-0.0377) (0.105,-0.0285) (0.125,-0.0192) (0.145,-0.0100) (0.166,-0.0008) (0.186,0.0084) (0.206,0.0177) (0.226,0.0269) (0.246,0.0361) (0.266,0.0454) (0.286,0.0546) (0.306,0.0638) (0.326,0.0730) (0.346,0.0823) (0.366,0.0915) (0.386,0.1007) (0.406,0.1100) (0.426,0.1192) (0.446,0.1284) (0.467,0.1376) (0.487,0.1469) (0.507,0.1561) (0.527,0.1653) (0.547,0.1746) (0.567,0.1838) (0.587,0.1930) (0.607,0.2022) (0.627,0.2115) (0.647,0.2207) (0.667,0.2299) (0.687,0.2392) (0.707,0.2484) (0.727,0.2576) (0.747,0.2668) (0.768,0.2761) (0.788,0.2853) (0.808,0.2945) (0.828,0.3038) (0.848,0.3130) (0.868,0.3222) (0.888,0.3314) (0.908,0.3407) (0.928,0.3499) (0.948,0.3591) (0.968,0.3684) (0.988,0.3776) (1.008,0.3868) (1.028,0.3960) (1.048,0.4053) (1.069,0.4145) (1.089,0.4237) (1.109,0.4330) (1.129,0.4422) (1.149,0.4514) (1.169,0.4606) (1.189,0.4699) (1.209,0.4791) (1.229,0.4883) (1.249,0.4976) (1.269,0.5068) (1.289,0.5160) (1.309,0.5252) (1.329,0.5345) (1.349,0.5437) (1.370,0.5529) (1.390,0.5622) (1.410,0.5714) (1.430,0.5806) (1.450,0.5898) (1.470,0.5991) (1.490,0.6083)};
\addplot[pc3!35, thin, forget plot] coordinates {(-1.500,0.9326) (-1.480,0.9153) (-1.460,0.8981) (-1.440,0.8808) (-1.420,0.8636) (-1.400,0.8463) (-1.380,0.8291) (-1.360,0.8118) (-1.339,0.7946) (-1.319,0.7773) (-1.299,0.7601) (-1.279,0.7428) (-1.259,0.7256) (-1.239,0.7084) (-1.219,0.6911) (-1.199,0.6739) (-1.179,0.6566) (-1.159,0.6394) (-1.139,0.6221) (-1.119,0.6049) (-1.099,0.5876) (-1.079,0.5704) (-1.059,0.5531) (-1.038,0.5359) (-1.018,0.5186) (-0.998,0.5014) (-0.978,0.4842) (-0.958,0.4669) (-0.938,0.4497) (-0.918,0.4324) (-0.898,0.4152) (-0.878,0.3979) (-0.858,0.3807) (-0.838,0.3634) (-0.818,0.3462) (-0.798,0.3289) (-0.778,0.3117) (-0.758,0.2944) (-0.737,0.2772) (-0.717,0.2599) (-0.697,0.2427) (-0.677,0.2255) (-0.657,0.2082) (-0.637,0.1910) (-0.617,0.1737) (-0.597,0.1565) (-0.577,0.1392) (-0.557,0.1220) (-0.537,0.1047) (-0.517,0.0875) (-0.497,0.0702) (-0.477,0.0530) (-0.457,0.0357) (-0.436,0.0185) (-0.416,0.0013) (-0.396,-0.0160) (-0.376,-0.0332) (-0.356,-0.0505) (-0.336,-0.0677) (-0.316,-0.0850) (-0.296,-0.1022) (-0.276,-0.1195) (-0.256,-0.1367) (-0.236,-0.1540) (-0.216,-0.1712) (-0.196,-0.1885) (-0.176,-0.2057) (-0.156,-0.2230) (-0.135,-0.2402) (-0.115,-0.2574) (-0.095,-0.2747) (-0.075,-0.2919)};
\addplot[pc4!35, thin, forget plot] coordinates {(0.045,-0.2942) (0.065,-0.2776) (0.085,-0.2611) (0.105,-0.2446) (0.125,-0.2281) (0.145,-0.2116) (0.166,-0.1950) (0.186,-0.1785) (0.206,-0.1620) (0.226,-0.1455) (0.246,-0.1289) (0.266,-0.1124) (0.286,-0.0959) (0.306,-0.0794) (0.326,-0.0629) (0.346,-0.0463) (0.366,-0.0298) (0.386,-0.0133) (0.406,0.0032) (0.426,0.0198) (0.446,0.0363) (0.467,0.0528) (0.487,0.0693) (0.507,0.0858) (0.527,0.1024) (0.547,0.1189) (0.567,0.1354) (0.587,0.1519) (0.607,0.1684) (0.627,0.1850) (0.647,0.2015) (0.667,0.2180) (0.687,0.2345) (0.707,0.2511) (0.727,0.2676) (0.747,0.2841) (0.768,0.3006) (0.788,0.3171) (0.808,0.3337) (0.828,0.3502) (0.848,0.3667) (0.868,0.3832) (0.888,0.3998) (0.908,0.4163) (0.928,0.4328) (0.948,0.4493) (0.968,0.4658) (0.988,0.4824) (1.008,0.4989) (1.028,0.5154) (1.048,0.5319) (1.069,0.5485) (1.089,0.5650) (1.109,0.5815) (1.129,0.5980) (1.149,0.6145) (1.169,0.6311) (1.189,0.6476) (1.209,0.6641) (1.229,0.6806) (1.249,0.6972) (1.269,0.7137) (1.289,0.7302) (1.309,0.7467) (1.329,0.7632) (1.349,0.7798) (1.370,0.7963) (1.390,0.8128) (1.410,0.8293) (1.430,0.8458) (1.450,0.8624) (1.470,0.8789) (1.490,0.8954)};
\addplot[pc5!35, thin, forget plot] coordinates {(-1.500,0.3808) (-1.480,0.3752) (-1.460,0.3696) (-1.440,0.3640) (-1.420,0.3585) (-1.400,0.3529) (-1.380,0.3473) (-1.360,0.3418) (-1.339,0.3362) (-1.319,0.3306) (-1.299,0.3251) (-1.279,0.3195) (-1.259,0.3139) (-1.239,0.3084) (-1.219,0.3028) (-1.199,0.2972) (-1.179,0.2917) (-1.159,0.2861) (-1.139,0.2805) (-1.119,0.2749) (-1.099,0.2694) (-1.079,0.2638) (-1.059,0.2582) (-1.038,0.2527) (-1.018,0.2471) (-0.998,0.2415) (-0.978,0.2360) (-0.958,0.2304) (-0.938,0.2248) (-0.918,0.2193) (-0.898,0.2137) (-0.878,0.2081) (-0.858,0.2026) (-0.838,0.1970) (-0.818,0.1914) (-0.798,0.1859) (-0.778,0.1803) (-0.758,0.1747) (-0.737,0.1691) (-0.717,0.1636) (-0.697,0.1580) (-0.677,0.1524) (-0.657,0.1469) (-0.637,0.1413) (-0.617,0.1357) (-0.597,0.1302) (-0.577,0.1246) (-0.557,0.1190) (-0.537,0.1135) (-0.517,0.1079) (-0.497,0.1023) (-0.477,0.0968) (-0.457,0.0912) (-0.436,0.0856) (-0.416,0.0800) (-0.396,0.0745) (-0.376,0.0689) (-0.356,0.0633) (-0.336,0.0578) (-0.316,0.0522) (-0.296,0.0466) (-0.276,0.0411) (-0.256,0.0355) (-0.236,0.0299) (-0.216,0.0244) (-0.196,0.0188) (-0.176,0.0132) (-0.156,0.0077) (-0.135,0.0021) (-0.115,-0.0035) (-0.095,-0.0090) (-0.075,-0.0146) (-0.055,-0.0202) (-0.035,-0.0258) (-0.015,-0.0313) (0.005,-0.0369) (0.025,-0.0425) (0.045,-0.0480) (0.065,-0.0536) (0.085,-0.0592) (0.105,-0.0647) (0.125,-0.0703) (0.145,-0.0759) (0.166,-0.0814) (0.186,-0.0870) (0.206,-0.0926) (0.226,-0.0981) (0.246,-0.1037) (0.266,-0.1093) (0.286,-0.1149) (0.306,-0.1204) (0.326,-0.1260) (0.346,-0.1316) (0.366,-0.1371) (0.386,-0.1427) (0.406,-0.1483) (0.426,-0.1538) (0.446,-0.1594) (0.467,-0.1650) (0.487,-0.1705) (0.507,-0.1761) (0.527,-0.1817) (0.547,-0.1872) (0.567,-0.1928) (0.587,-0.1984) (0.607,-0.2040) (0.627,-0.2095) (0.647,-0.2151) (0.667,-0.2207) (0.687,-0.2262) (0.707,-0.2318) (0.727,-0.2374) (0.747,-0.2429) (0.768,-0.2485) (0.788,-0.2541) (0.808,-0.2596) (0.828,-0.2652) (0.848,-0.2708) (0.868,-0.2763) (0.888,-0.2819) (0.908,-0.2875) (0.928,-0.2930) (0.948,-0.2986)};
\addplot[pc6!35, thin, forget plot] coordinates {(-1.500,0.6707) (-1.480,0.6600) (-1.460,0.6494) (-1.440,0.6387) (-1.420,0.6280) (-1.400,0.6173) (-1.380,0.6067) (-1.360,0.5960) (-1.339,0.5853) (-1.319,0.5746) (-1.299,0.5639) (-1.279,0.5533) (-1.259,0.5426) (-1.239,0.5319) (-1.219,0.5212) (-1.199,0.5105) (-1.179,0.4999) (-1.159,0.4892) (-1.139,0.4785) (-1.119,0.4678) (-1.099,0.4572) (-1.079,0.4465) (-1.059,0.4358) (-1.038,0.4251) (-1.018,0.4144) (-0.998,0.4038) (-0.978,0.3931) (-0.958,0.3824) (-0.938,0.3717) (-0.918,0.3610) (-0.898,0.3504) (-0.878,0.3397) (-0.858,0.3290) (-0.838,0.3183) (-0.818,0.3076) (-0.798,0.2970) (-0.778,0.2863) (-0.758,0.2756) (-0.737,0.2649) (-0.717,0.2543) (-0.697,0.2436) (-0.677,0.2329) (-0.657,0.2222) (-0.637,0.2115) (-0.617,0.2009) (-0.597,0.1902) (-0.577,0.1795) (-0.557,0.1688) (-0.537,0.1581) (-0.517,0.1475) (-0.497,0.1368) (-0.477,0.1261) (-0.457,0.1154) (-0.436,0.1048) (-0.416,0.0941) (-0.396,0.0834) (-0.376,0.0727) (-0.356,0.0620) (-0.336,0.0514) (-0.316,0.0407) (-0.296,0.0300) (-0.276,0.0193) (-0.256,0.0086) (-0.236,-0.0020) (-0.216,-0.0127) (-0.196,-0.0234) (-0.176,-0.0341) (-0.156,-0.0448) (-0.135,-0.0554) (-0.115,-0.0661) (-0.095,-0.0768) (-0.075,-0.0875) (-0.055,-0.0981) (-0.035,-0.1088) (-0.015,-0.1195) (0.005,-0.1302) (0.025,-0.1409) (0.045,-0.1515) (0.065,-0.1622) (0.085,-0.1729) (0.105,-0.1836) (0.125,-0.1943) (0.145,-0.2049) (0.166,-0.2156) (0.186,-0.2263) (0.206,-0.2370) (0.226,-0.2476) (0.246,-0.2583) (0.266,-0.2690) (0.286,-0.2797) (0.306,-0.2904)};
\addplot[pc7!35, thin, forget plot] coordinates {(0.326,-0.2932) (0.346,-0.2700) (0.366,-0.2467) (0.386,-0.2234) (0.406,-0.2002) (0.426,-0.1769) (0.446,-0.1536) (0.467,-0.1304) (0.487,-0.1071) (0.507,-0.0838) (0.527,-0.0606) (0.547,-0.0373) (0.567,-0.0140) (0.587,0.0092) (0.607,0.0325) (0.627,0.0558) (0.647,0.0790) (0.667,0.1023) (0.687,0.1256) (0.707,0.1488) (0.727,0.1721) (0.747,0.1954) (0.768,0.2186) (0.788,0.2419) (0.808,0.2652) (0.828,0.2884) (0.848,0.3117) (0.868,0.3350) (0.888,0.3582) (0.908,0.3815) (0.928,0.4047) (0.948,0.4280) (0.968,0.4513) (0.988,0.4745) (1.008,0.4978) (1.028,0.5211) (1.048,0.5443) (1.069,0.5676) (1.089,0.5909) (1.109,0.6141) (1.129,0.6374) (1.149,0.6607) (1.169,0.6839) (1.189,0.7072) (1.209,0.7305) (1.229,0.7537) (1.249,0.7770) (1.269,0.8003) (1.289,0.8235) (1.309,0.8468) (1.329,0.8701) (1.349,0.8933) (1.370,0.9166) (1.390,0.9399) (1.410,0.9631) (1.430,0.9864) (1.450,1.0097) (1.470,1.0329) (1.490,1.0562)};
\addplot[pc8!35, thin, forget plot] coordinates {(-1.500,1.0830) (-1.480,1.0588) (-1.460,1.0345) (-1.440,1.0102) (-1.420,0.9859) (-1.400,0.9616) (-1.380,0.9373) (-1.360,0.9131) (-1.339,0.8888) (-1.319,0.8645) (-1.299,0.8402) (-1.279,0.8159) (-1.259,0.7916) (-1.239,0.7674) (-1.219,0.7431) (-1.199,0.7188) (-1.179,0.6945) (-1.159,0.6702) (-1.139,0.6459) (-1.119,0.6217) (-1.099,0.5974) (-1.079,0.5731) (-1.059,0.5488) (-1.038,0.5245) (-1.018,0.5002) (-0.998,0.4759) (-0.978,0.4517) (-0.958,0.4274) (-0.938,0.4031) (-0.918,0.3788) (-0.898,0.3545) (-0.878,0.3302) (-0.858,0.3060) (-0.838,0.2817) (-0.818,0.2574) (-0.798,0.2331) (-0.778,0.2088) (-0.758,0.1845) (-0.737,0.1603) (-0.717,0.1360) (-0.697,0.1117) (-0.677,0.0874) (-0.657,0.0631) (-0.637,0.0388) (-0.617,0.0146) (-0.597,-0.0097) (-0.577,-0.0340) (-0.557,-0.0583) (-0.537,-0.0826) (-0.517,-0.1069) (-0.497,-0.1311) (-0.477,-0.1554) (-0.457,-0.1797) (-0.436,-0.2040) (-0.416,-0.2283) (-0.396,-0.2526) (-0.376,-0.2768)};
\addplot[pc9!35, thin, forget plot] coordinates {(-0.216,-0.2906) (-0.196,-0.2789) (-0.176,-0.2671) (-0.156,-0.2554) (-0.135,-0.2437) (-0.115,-0.2319) (-0.095,-0.2202) (-0.075,-0.2085) (-0.055,-0.1967) (-0.035,-0.1850) (-0.015,-0.1732) (0.005,-0.1615) (0.025,-0.1498) (0.045,-0.1380) (0.065,-0.1263) (0.085,-0.1146) (0.105,-0.1028) (0.125,-0.0911) (0.145,-0.0793) (0.166,-0.0676) (0.186,-0.0559) (0.206,-0.0441) (0.226,-0.0324) (0.246,-0.0207) (0.266,-0.0089) (0.286,0.0028) (0.306,0.0146) (0.326,0.0263) (0.346,0.0380) (0.366,0.0498) (0.386,0.0615) (0.406,0.0732) (0.426,0.0850) (0.446,0.0967) (0.467,0.1085) (0.487,0.1202) (0.507,0.1319) (0.527,0.1437) (0.547,0.1554) (0.567,0.1671) (0.587,0.1789) (0.607,0.1906) (0.627,0.2024) (0.647,0.2141) (0.667,0.2258) (0.687,0.2376) (0.707,0.2493) (0.727,0.2610) (0.747,0.2728) (0.768,0.2845) (0.788,0.2963) (0.808,0.3080) (0.828,0.3197) (0.848,0.3315) (0.868,0.3432) (0.888,0.3549) (0.908,0.3667) (0.928,0.3784) (0.948,0.3902) (0.968,0.4019) (0.988,0.4136) (1.008,0.4254) (1.028,0.4371) (1.048,0.4488) (1.069,0.4606) (1.089,0.4723) (1.109,0.4840) (1.129,0.4958) (1.149,0.5075) (1.169,0.5193) (1.189,0.5310) (1.209,0.5427) (1.229,0.5545) (1.249,0.5662) (1.269,0.5779) (1.289,0.5897) (1.309,0.6014) (1.329,0.6132) (1.349,0.6249) (1.370,0.6366) (1.390,0.6484) (1.410,0.6601) (1.430,0.6718) (1.450,0.6836) (1.470,0.6953) (1.490,0.7071)};
\addplot[grayax!65, dashed, thin] coordinates {(-1.500,1.1250) (-1.480,1.0951) (-1.460,1.0656) (-1.440,1.0365) (-1.420,1.0078) (-1.400,0.9795) (-1.380,0.9516) (-1.360,0.9242) (-1.339,0.8971) (-1.319,0.8704) (-1.299,0.8441) (-1.279,0.8183) (-1.259,0.7928) (-1.239,0.7677) (-1.219,0.7431) (-1.199,0.7188) (-1.179,0.6949) (-1.159,0.6715) (-1.139,0.6484) (-1.119,0.6258) (-1.099,0.6035) (-1.079,0.5817) (-1.059,0.5602) (-1.038,0.5392) (-1.018,0.5186) (-0.998,0.4983) (-0.978,0.4785) (-0.958,0.4591) (-0.938,0.4400) (-0.918,0.4214) (-0.898,0.4032) (-0.878,0.3854) (-0.858,0.3680) (-0.838,0.3509) (-0.818,0.3343) (-0.798,0.3181) (-0.778,0.3023) (-0.758,0.2869) (-0.737,0.2719) (-0.717,0.2573) (-0.697,0.2431) (-0.677,0.2293) (-0.657,0.2159) (-0.637,0.2030) (-0.617,0.1904) (-0.597,0.1782) (-0.577,0.1664) (-0.557,0.1550) (-0.537,0.1441) (-0.517,0.1335) (-0.497,0.1233) (-0.477,0.1136) (-0.457,0.1042) (-0.436,0.0952) (-0.416,0.0867) (-0.396,0.0785) (-0.376,0.0708) (-0.356,0.0634) (-0.336,0.0565) (-0.316,0.0499) (-0.296,0.0438) (-0.276,0.0381) (-0.256,0.0327) (-0.236,0.0278) (-0.216,0.0233) (-0.196,0.0191) (-0.176,0.0154) (-0.156,0.0121) (-0.135,0.0092) (-0.115,0.0067) (-0.095,0.0045) (-0.075,0.0028) (-0.055,0.0015) (-0.035,0.0006) (-0.015,0.0001) (0.005,0.0000) (0.025,0.0003) (0.045,0.0010) (0.065,0.0021) (0.085,0.0036) (0.105,0.0055) (0.125,0.0079) (0.145,0.0106) (0.166,0.0137) (0.186,0.0172) (0.206,0.0212) (0.226,0.0255) (0.246,0.0302) (0.266,0.0353) (0.286,0.0409) (0.306,0.0468) (0.326,0.0532) (0.346,0.0599) (0.366,0.0671) (0.386,0.0746) (0.406,0.0826) (0.426,0.0909) (0.446,0.0997) (0.467,0.1088) (0.487,0.1184) (0.507,0.1284) (0.527,0.1387) (0.547,0.1495) (0.567,0.1607) (0.587,0.1723) (0.607,0.1842) (0.627,0.1966) (0.647,0.2094) (0.667,0.2226) (0.687,0.2362) (0.707,0.2502) (0.727,0.2646) (0.747,0.2794) (0.768,0.2946) (0.788,0.3102) (0.808,0.3262) (0.828,0.3426) (0.848,0.3594) (0.868,0.3766) (0.888,0.3942) (0.908,0.4123) (0.928,0.4307) (0.948,0.4495) (0.968,0.4687) (0.988,0.4884) (1.008,0.5084) (1.028,0.5288) (1.048,0.5497) (1.069,0.5709) (1.089,0.5926) (1.109,0.6146) (1.129,0.6371) (1.149,0.6599) (1.169,0.6832) (1.189,0.7068) (1.209,0.7309) (1.229,0.7553) (1.249,0.7802) (1.269,0.8055) (1.289,0.8311) (1.309,0.8572) (1.329,0.8837) (1.349,0.9106) (1.370,0.9379) (1.390,0.9655) (1.410,0.9936) (1.430,1.0221) (1.450,1.0510) (1.470,1.0803) (1.490,1.1100)};
\addlegendentry{$\rho(t)=\tfrac{1}{2}t^2$}
\addplot[greenc, thick, line cap=round] coordinates {(-1.500,1.0830) (-1.480,1.0588) (-1.460,1.0345) (-1.440,1.0102) (-1.420,0.9859) (-1.400,0.9616) (-1.380,0.9373) (-1.360,0.9131) (-1.339,0.8888) (-1.319,0.8645) (-1.299,0.8402) (-1.279,0.8159) (-1.259,0.7916) (-1.239,0.7674) (-1.219,0.7431) (-1.199,0.7188) (-1.179,0.6945) (-1.159,0.6702) (-1.139,0.6459) (-1.119,0.6217) (-1.099,0.5974) (-1.079,0.5731) (-1.059,0.5531) (-1.038,0.5359) (-1.018,0.5186) (-0.998,0.5014) (-0.978,0.4842) (-0.958,0.4669) (-0.938,0.4497) (-0.918,0.4324) (-0.898,0.4152) (-0.878,0.3979) (-0.858,0.3807) (-0.838,0.3634) (-0.818,0.3462) (-0.798,0.3289) (-0.778,0.3117) (-0.758,0.2944) (-0.737,0.2772) (-0.717,0.2599) (-0.697,0.2436) (-0.677,0.2329) (-0.657,0.2222) (-0.637,0.2115) (-0.617,0.2009) (-0.597,0.1902) (-0.577,0.1795) (-0.557,0.1688) (-0.537,0.1581) (-0.517,0.1475) (-0.497,0.1368) (-0.477,0.1261) (-0.457,0.1154) (-0.436,0.1048) (-0.416,0.0941) (-0.396,0.0834) (-0.376,0.0767) (-0.356,0.0703) (-0.336,0.0638) (-0.316,0.0573) (-0.296,0.0509) (-0.276,0.0444) (-0.256,0.0380) (-0.236,0.0315) (-0.216,0.0251) (-0.196,0.0188) (-0.176,0.0132) (-0.156,0.0077) (-0.135,0.0021) (-0.115,0.0030) (-0.095,0.0041) (-0.075,0.0051) (-0.055,0.0062) (-0.035,0.0073) (-0.015,0.0084) (0.005,0.0095) (0.025,0.0106) (0.045,0.0116) (0.065,0.0127) (0.085,0.0138) (0.105,0.0149) (0.125,0.0160) (0.145,0.0171) (0.166,0.0181) (0.186,0.0192) (0.206,0.0203) (0.226,0.0269) (0.246,0.0361) (0.266,0.0454) (0.286,0.0546) (0.306,0.0638) (0.326,0.0730) (0.346,0.0823) (0.366,0.0915) (0.386,0.1007) (0.406,0.1100) (0.426,0.1192) (0.446,0.1284) (0.467,0.1376) (0.487,0.1469) (0.507,0.1561) (0.527,0.1653) (0.547,0.1746) (0.567,0.1838) (0.587,0.1930) (0.607,0.2022) (0.627,0.2115) (0.647,0.2207) (0.667,0.2299) (0.687,0.2392) (0.707,0.2511) (0.727,0.2676) (0.747,0.2841) (0.768,0.3006) (0.788,0.3171) (0.808,0.3337) (0.828,0.3502) (0.848,0.3667) (0.868,0.3832) (0.888,0.3998) (0.908,0.4163) (0.928,0.4328) (0.948,0.4493) (0.968,0.4658) (0.988,0.4824) (1.008,0.4989) (1.028,0.5211) (1.048,0.5443) (1.069,0.5676) (1.089,0.5909) (1.109,0.6141) (1.129,0.6374) (1.149,0.6607) (1.169,0.6839) (1.189,0.7072) (1.209,0.7305) (1.229,0.7537) (1.249,0.7770) (1.269,0.8003) (1.289,0.8235) (1.309,0.8468) (1.329,0.8701) (1.349,0.8933) (1.370,0.9166) (1.390,0.9399) (1.410,0.9631) (1.430,0.9864) (1.450,1.0097) (1.470,1.0329) (1.490,1.0562)};
\addlegendentry{$f_{\varepsilon,d}(t,0)$\;(reconstruction)}
\addplot[pc8, mark=*, mark size=1.6pt, only marks, forget plot] coordinates {(-1.20,0.7200)};
\addplot[pc3, mark=*, mark size=1.6pt, only marks, forget plot] coordinates {(-0.70,0.2450)};
\addplot[pc0, mark=*, mark size=1.6pt, only marks, forget plot] coordinates {(-0.20,0.0200)};
\addplot[pc1, mark=*, mark size=1.6pt, only marks, forget plot] coordinates {(0.20,0.0200)};
\addplot[pc4, mark=*, mark size=1.6pt, only marks, forget plot] coordinates {(0.70,0.2450)};
\addplot[pc7, mark=*, mark size=1.6pt, only marks, forget plot] coordinates {(1.20,0.7200)};
\end{axis}
\end{tikzpicture}
\caption{Restriction of the affine input coordinates
$x^{(0)}_m(t)=\langle p_m,(t,0)\rangle_H+q_m$
to the slice $x=(t,0)$.  The faint coloured lines are the affine pieces
generated by the directions in~\subref{fig:disk}; the green curve is their
pointwise maximum, implementing the outer max operation in
\eqref{eq:I_dualized_discretized__hellooo} and the max-pooling structure
in~\eqref{eq:maxpooling}.}
\label{fig:recon}
\end{subfigure}
 
\caption{Dual-ball discretization and reconstruction formula.
\subref{fig:disk}~The $\eta$-net $\{p_m\}$ of
$V_2\cap\overline{B}_H(0,L)$ discretizes the supremum in the dual
form~\eqref{eq:I_dualized_you_bby}; each $p_m$ fixes the slope of one affine
coordinate $x^{(0)}_m=\langle p_m,x\rangle_H+q_m$ in Definition~\ref{defn:CNFs}.
\subref{fig:recon}~On the slice $x=(t,0)$, these affine pieces appear as faint
coloured lines, while their pointwise maximum gives
$f_{\varepsilon,d}$ in~\eqref{eq:I_dualized_discretized__hellooo}. Convexity
comes from the max-envelope structure, and Lipschitzness from
$\|p_m\|_H\le L$.}
\end{figure}

It is natural to ask whether there is a larger, \textbf{more structured}, trainable class containing every network of the form~\eqref{eq:relu_representability_formula}, with the property that, regardless of the data used for training and regardless of the admissible parameters selected by the learning algorithm, every realized network defines a convex and Lipschitz functional on $H$.  Our final result provides precisely such a certificate: it shows that the CNF class from Definition~\ref{defn:CNFs} is automatically convexity- and Lipschitzness-preserving.  Theorems~\ref{thrm:relu_MLP_computation} and~\ref{thrm:reconstruction} have already guaranteed the universality of our deep learning models within the broader (concept) class of convex Lipschitz functionals.

Critically, the \textit{previous} result, Theorem~\ref{thrm:relu_MLP_computation}, shows that the function~\eqref{eq:I_dualized_discretized__hellooo} can be exactly computed/represented by some $\operatorname{ReLU}$-MLP with a particular configuration; however, many $\operatorname{ReLU}$-MLP configurations are clearly non-convex and violate the structural conditions we seek, primarily convexity.  In this context, the next result, Theorem~\ref{thrm:certificate}, not only guarantees that this function can be computed by a CNF, but moreover that every CNF has the basic structural properties we seek.  Thus, no training algorithm can break that structure, in stark contrast with the standard, relatively unstructured and perhaps overly flexible, $\operatorname{ReLU}$-MLP model considered above.

\begin{theorem}[Certifiably Convex and Lipschitz]
\label{thrm:certificate}
Let $H$ be a real Hilbert space and let $\Phi:H\to\mathbb{R}$ be a convex neural functional in the sense of Definition~\ref{defn:CNFs}.  Then $\Phi$ is convex and Lipschitz.
\hfill\\
Moreover, every $f_{\varepsilon,d}$ in Theorem~\ref{thrm:relu_MLP_computation}, cf.~\eqref{eq:relu_representability_formula}, is exactly computable by some CNF.
\end{theorem}

Before concluding our analysis, we consider one toy validation showing that our main result does indeed bridge the theory-practice barrier.
\subsection{Does it Work? A Toy Numerical Validation}
\label{s:Applications}
\subsubsection*{Toy experiment.}
We generate training data from a randomly initialized convex $\operatorname{ReLU}$-MLP with two hidden layers of width $500$ and nonnegative hidden/output weights.  We then train a CNF with approximately one third of the target model's parameters.  In the one-dimensional experiment shown in Figure~\ref{fig:cnf_toy}, the target network has $252{,}001$ parameters, while the CNF has $81{,}402$ trainable parameters, giving a CNF/target parameter ratio of $0.323$.  The CNF is trained on $1000$ samples for $200$ gradient-descent iterations and evaluated on $200$ test points.

\begin{figure}[ht]
    \centering
    \includegraphics[width=0.55\linewidth]{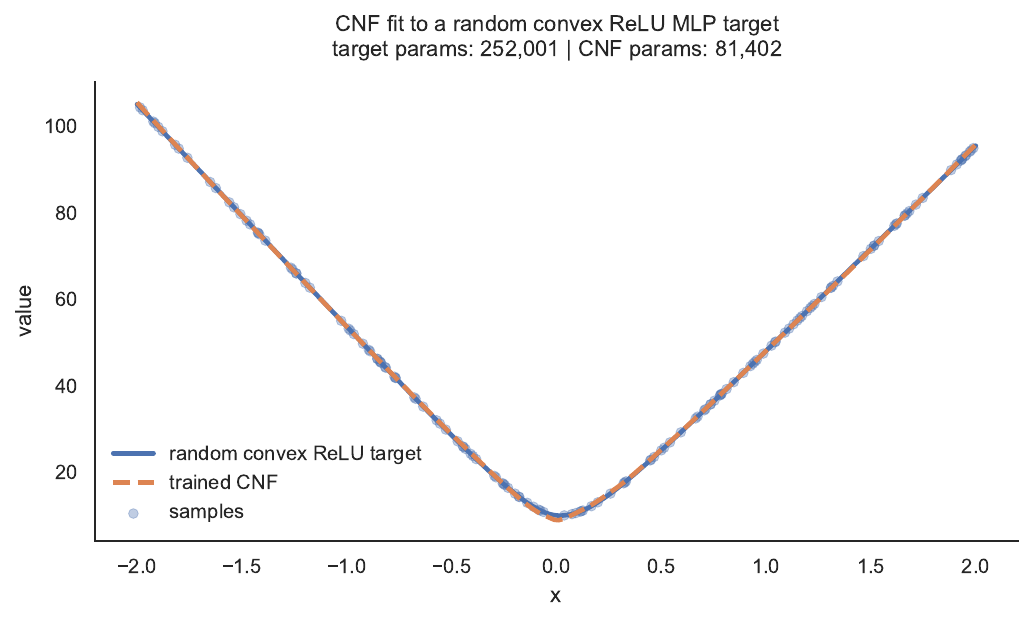}
    \caption{
    CNF approximation of a randomly generated convex $\operatorname{ReLU}$-MLP target in dimension $d_{\mathrm{in}}=1$.  The target network has two hidden layers of width $500$ and $252{,}001$ parameters, while the CNF has $81{,}402$ trainable parameters, corresponding to a parameter ratio of $0.323$.  The CNF is trained on $1000$ samples for $200$ optimization iterations.  Since all CNF iterates remain in the admissible architecture class, convexity is preserved throughout training.
    }
    \label{fig:cnf_toy}
\end{figure}

\subsubsection*{Dimensional ablation.}
We ablate the effect of dimension by repeating the previous experiment while varying the input dimension over $d_{\mathrm{in}}\in\{1,20,50,100\}$, using five independent runs for each dimension.  Table~\ref{tab:cnf_dimensional_experiment} reports mean $\pm$ standard deviation.  The Jensen column reports the maximum sampled Jensen gap $\Phi(tx+(1-t)y)-t\Phi(x)-(1-t)\Phi(y)$; values below zero indicate that all sampled tests satisfied Jensen convexity with a strict margin.
The small positive value in dimension $1$ is at numerical precision scale and
is consistent with the exact architectural convexity certificate.

\begin{table}[ht]
\centering
\caption{Dimensional Ablation.  For each dimension, we train a CNF on samples generated by a larger random convex $\operatorname{ReLU}$-MLP target.  Results report mean $\pm$ standard deviation over five independent runs.
}
\label{tab:cnf_dimensional_experiment}
\begin{adjustbox}{max width=\linewidth}
\begin{tabular}{lrrrr}
\toprule
 & $d_{\mathrm{in}}=1$ & $d_{\mathrm{in}}=20$ & $d_{\mathrm{in}}=50$ & $d_{\mathrm{in}}=100$ \\
\midrule
% Target params & 252001 & 261501 & 276501 & 301501 \\
% CNF params & 81402 & 83302 & 86302 & 91302 \\
Param. ratio & $0.323\pm 0.000$ & $0.319\pm 0.000$ & $0.312\pm 0.000$ & $0.303\pm 0.000$ \\
Train MSE & $1.940e+02\pm 1.425e+02$ & $4.208e+02\pm 5.119e+02$ & $4.959e+01\pm 6.616e+00$ & $3.914e+01\pm 1.331e+00$ \\
Test MSE & $1.801e+02\pm 1.292e+02$ & $4.511e+02\pm 5.514e+02$ & $6.197e+01\pm 6.236e+00$ & $6.565e+01\pm 9.980e+00$ \\
CNF Jensen gap & $2.747e-05\pm 6.824e-06$ & $-6.683e-03\pm 5.112e-03$ & $-1.406e-02\pm 1.860e-02$ & $-1.048e-02\pm 7.872e-03$ \\
\bottomrule
\end{tabular}
\end{adjustbox}
\end{table}

\section{Conclusion}
\label{s:Conclusion}

We have shown that finite samples of an $L$-Lipschitz convex functional $\rho:C\to\mathbb{R}$ on a compact convex subset of a separable Hilbert space admit a structure-preserving reconstruction theory.  First, we constructed an explicit finite-sample formula which is convex, $L$-Lipschitz, and uniformly accurate on $C$ (Theorem~\ref{thrm:reconstruction}).  Then, we showed that this reconstruction is exactly computable by sparsely connected $\operatorname{ReLU}$-MLPs using only finitely many Hilbert-space measurements of the form $\langle b,\cdot\rangle_H$ (Theorem~\ref{thrm:relu_MLP_computation}).  Finally, we certified that this reconstruction belongs to the broader trainable class of convex neural functionals (CNFs), every admissible parameter configuration of which is automatically convex and Lipschitz (Theorem~\ref{thrm:certificate}).  Our toy numerical experiments illustrate that CNFs can learn randomly generated convex $\operatorname{ReLU}$-MLP targets while preserving convexity throughout training.

\subsection{Future Work}
It would be interesting to consider specialized applications of our general tools in risk management, where one seeks to learn a convex, or even coherent, risk measure from finitely many risk measurements at finitely many random variables, and in mathematical finance, where one seeks to learn a super-hedging price from finitely many contingent claims.  On the theoretical side, we will extend the present Hilbert-space theory to the separable, and likely reflexive, Banach space setting with applications to spaces such as $L^p_{\mathbb{P}}(\mathcal{F})$ over some probability space $(\Omega,\mathcal{F},\mathbb{P})$ for finite $1<p$.

\section*{Acknowledgements and Funding}
\label{s:AckFund}
A.\ Kratsios acknowledges financial support from an NSERC Discovery Grant No.\ RGPIN-2023-04482 and No.\ DGECR-2023-00230.  
The authors acknowledge that resources used in preparing this research were provided, in part, by the Province of Ontario, the Government of Canada through CIFAR, and companies sponsoring the Vector Institute\footnote{\href{https://vectorinstitute.ai/partnerships/current-partners/}{https://vectorinstitute.ai/partnerships/current-partners/}}.

\section{Proofs}
\label{s:Proofs}
We now prove our main results.
\subsection{{Proof of Reconstruction Result (Theorem~\ref{thrm:reconstruction})}}
\label{s:Proof__thrm:reconstruction}
We now state the technical version of our main result before moving to its proof.
\begin{theorem}
\label{thrm:reconstruction__full}
Let $L>0$.  Let $H$ be a separable Hilbert space, let $C\subseteq H$ be compact, convex, and with
$D\eqdef \operatorname{diam}(C)>0$. 
\hfill\\
For any convex, $L$-Lipschitz map $\rho:C\to\mathbb{R}$ and any $\varepsilon>0$, set
$
    \delta\eqdef \frac{\varepsilon}{4L}
$, $
    \alpha\eqdef \frac{\varepsilon}{16L}
$, and $
    \eta\eqdef \frac{\varepsilon}{8D}
$.
\hfill\\
Let $\mathbb{X}_N\eqdef\{\xi_n\}_{n=1}^N\subseteq C$ be a $\delta$-net of $C$ with
$N=N_\delta(C)$.  Then there exist $d\in\mathbb{N}_0$, a $d$-dimensional linear subspace
$V_d\subseteq H$, and an $\eta$-net
$p^\eta\eqdef\{p_m\}_{m=1}^M$ of $V_d\cap \overline{B}_H(0,L)$ such that the induced map
$f_{\varepsilon,d}:H\to \mathbb{R}$ defined by
\begin{equation}
\label{eq:I_dualized_discretized__hellooo2}
    f_{\varepsilon,d}(x)
\eqdef
    \max_{m=1,\dots,M}
    \,
    \Big\{
        \langle p_m,x\rangle
        +
        \min_{1\leq n\leq N}
        \bigl(
            \rho(\xi_n)-\langle p_m,\xi_n\rangle
        \bigr)
    \Big\}
\mbox{ (for each } x\in H)
\end{equation}
is convex, $L$-Lipschitz, and satisfies 
\begin{equation}
\label{eq:das_uniformli_estimatoski_samples2}
    \sup_{\xi\in C}
    \,
        \bigl|
            \rho(\xi)
            -
            f_{\varepsilon,d}(\xi)
        \bigr|
    <
        \varepsilon.
\end{equation}
Moreover,
$
    N
    =
    N_{\varepsilon/(4L)}(C)
$, 
$
    d
    =
    d_C\,(\frac{\varepsilon}{16L})
    \leq
    N_{\varepsilon/(16L)}(C)
$, 
and $
    M
    \leq
    (
        1+\tfrac{16LD}{\varepsilon}
    )^d
$; where
\[
    d_C(\alpha)
    \eqdef
    \min
    \biggl\{
        k\in\mathbb{N}_0:
        \exists\, V_k\subseteq H \text{ linear subspace with } \dim(V_k)=k
        \text{ and }
        \sup_{\xi\in C}\operatorname{dist}(\xi,V_k)\leq \alpha
    \biggr\}
\]
for any $\alpha>0$.
\end{theorem}

We prove our result in three steps.  First, we obtain an exact but in-computable primal form.  Next, we dualize and discretize it in a way which we preserve both convexity and Lipschitz regularity.  Finally, we estimate the remaining complexity parameters $d$ and $M$ (as $N$ is computed in the second step).
\subsubsection{Step 1 - Infinitely Parameterized, Exact, Primal Form}
In what follows, we denote the Fenchel (convex) dual of a function on $f:H\to \mathbb{R}$ by $f^{\star}$; its convex envelope is thus $f^{\star\,\star}$.  Note that, when $f$ is convex then by the Fenchel duality theorem $f^{\star\,\star}=f$.
The following is a low-regularity version of the main result in~\cite{azagra2018explicit}; which conveniently admits a simpler form in the finite data case.
\begin{lemma}[Convexified Whitney-McShane Extension]
\label{lem:convexification_finite_whitney_mcshane}
Let $H$ be a Hilbert space, let $C\subseteq H$ be non-empty and convex, and let
$\rho:C\to\mathbb{R}$ be convex and $L$-Lipschitz.  Fix $N\in \mathbb{N}_+$ and
$\mathbb{X}_N\eqdef\{\xi_n\}_{n=1}^N\subseteq C$, and define the upper
McShane extension, cf.~\cite{cobzacs2019lipschitz}, for every $x\in H$ by
\begin{equation}
\label{eq:upper_extension}
    f_N\eqdef g_N^{\star\,\star}
\mbox{ and }
    g_N(x)
    \eqdef
    \min_{1\leq n\leq N}
    \big\{
        \rho(\xi_n)+L\|x-\xi_n\|
    \big\}.
\end{equation}
Then $f_N$ is convex, $L$-Lipschitz, and interpolates the data:
$
    f_N(\xi_n)=\rho(\xi_n)
$ for all $n=1,\dots,N$.
\hfill\\
Moreover, $f_N$ admits the semi-closed-form expression: 
\begin{equation}
\label{eq:semi_closedform}
    f_N(x)
=
    \min_{\lambda\in\Delta_N}
    \,
    \biggl\{
        \sum_{n=1}^N \lambda_n\rho(\xi_n)
        +
        L\Big\|
            x-\sum_{n=1}^N\lambda_n\xi_n
        \Big\|
    \biggr\}
.
\end{equation}
\end{lemma}
\begin{proof}[{Proof of Lemma~\ref{lem:convexification_finite_whitney_mcshane}}]
For each $n=1,\dots,N$ abbreviate $y_n\eqdef \rho(\xi_n)$.
Since $\rho$ is $L$-Lipschitz on $C$, for every $m,n\in\{1,\dots,N\}$,
\[
    \rho(\xi_m)\leq \rho(\xi_n)+L\|\xi_m-\xi_n\|.
\]
Hence $g_N$ interpolates the data:
$
    g_N(\xi_m)=\rho(\xi_m)
$
for every $m=1,\dots,N$.  Moreover, $g_N$ is $L$-Lipschitz, as a finite
minimum of $L$-Lipschitz functions.
We construct a convex minorant of $g_N$ which already interpolates the data.
Let $\Delta_N$ denote the simplex in $\mathbb{R}^N$ and define
\begin{equation}
\label{eq:R_Ndefo}
    R_N(x)\eqdef
    \inf_{\lambda\in\Delta_N}
    \Big\{
        \sum_{n=1}^N \lambda_n\rho(\xi_n)
        +L\left\|x-\sum_{n=1}^N \lambda_n\xi_n\right\|
    \Big\}
\end{equation}
for each $x\in H$.
\hfill\\
Note that $R_N$ is $L$-Lipschitz, hence continuous.  Indeed, for every
$x,y\in H$ and every $\lambda\in\Delta_N$,
\[
    \sum_{n=1}^N\lambda_n\rho(\xi_n)
    +
    L\left\|x-\sum_{n=1}^N\lambda_n\xi_n\right\|
    \leq
    \sum_{n=1}^N\lambda_n\rho(\xi_n)
    +
    L\left\|y-\sum_{n=1}^N\lambda_n\xi_n\right\|
    +
    L\|x-y\|.
\]
Taking the infimum over $\lambda\in\Delta_N$ gives
$R_N(x)\leq R_N(y)+L\|x-y\|$, and exchanging $x$ and $y$ gives
$|R_N(x)-R_N(y)|\leq L\|x-y\|$.
\hfill\\
The function inside the infimum in~\eqref{eq:R_Ndefo} is jointly convex in $(x,\lambda)$; hence
$R_N$ is convex.  Taking $\lambda=e_n$ implies that $R_N\leq g_N$, since for every $n=1,\dots,N$ we have
\[
    R_N(x)\leq \rho(\xi_n)+L\|x-\xi_n\|
.
\]
It remains to verify interpolation.  Fix $m\in\{1,\dots,N\}$ and
$\lambda\in\Delta_N$.  Since $C$ is convex,
$
    y_{\lambda}\eqdef \sum_{n=1}^N\lambda_n\xi_n\in C
$.
By convexity of $\rho$,
$
    \rho(y_{\lambda})
    \leq
    \sum_{n=1}^N\lambda_n\rho(\xi_n)
$.
By the $L$-Lipschitzness of $\rho$,
$
    \rho(\xi_m)
    \leq
    \rho(y_{\lambda})+L\|\xi_m-y_{\lambda}\|
$.
Consequently,
\[
    \rho(\xi_m)
    \leq
    \sum_{n=1}^N\lambda_n\rho(\xi_n)
    +L\left\|\xi_m-\sum_{n=1}^N\lambda_n\xi_n\right\|.
\]
Taking the infimum over $\lambda\in\Delta_N$ gives
$\rho(\xi_m)\leq R_N(\xi_m)$.  The reverse inequality follows by taking
$\lambda=e_m$.  Thus, for each $m=1,\dots,N$ we have that
$
    R_N(\xi_m)=\rho(\xi_m)
$.  
Since $R_N$ is continuous, convex, and $R_N\leq g_N$, while
$f_N=g_N^{\star\star}$ is the largest lower semicontinuous convex minorant of
$g_N$, we have
$
    R_N\leq f_N\leq g_N
$.
Evaluating at $\xi_m$ gives
\[
    \rho(\xi_m)
    =
    R_N(\xi_m)
    \leq
    f_N(\xi_m)
    \leq
    g_N(\xi_m)
    =
    \rho(\xi_m).
\]
Therefore $f_N(\xi_m)=\rho(\xi_m)$ for every $m=1,\dots,N$.

Finally, we show that $f_N$ is $L$-Lipschitz.  By the Fenchel-Moreau
representation, $f_N=g_N^{\star\star}$ is the supremum of all continuous affine
minorants of $g_N$.  Thus
\[
    f_N(x)
    =
    \sup
    \big\{
        \ell(x):
        \ell \text{ is continuous affine and } \ell\leq g_N \text{ on } H
    \big\}
\]
where $\le$ is meant in the pointwise sense; i.e.\ $f\le g$ if $f(x)\le g(x)$ for all $x\in H$.
Let $\ell(x)=\langle p,x\rangle+b$ be such a minorant.  We claim that
$\|p\|\leq L$.  If $p=0$, this is immediate.  Otherwise, set
$u\eqdef p/\|p\|$.  Since $g_N$ is $L$-Lipschitz, for every $x\in H$ and
$t>0$,
\begin{equation}
\label{eq:limiting_setup}
    \ell(x)+t\|p\|
    =
    \ell(x+tu)
    \leq
    g_N(x+tu)
    \leq
    g_N(x)+Lt.
\end{equation}
Hence
$
    t(\|p\|-L)\leq g_N(x)-\ell(x)
$.
Letting $t\to\infty$ in~\eqref{eq:limiting_setup} implies that $\|p\|\leq L$.
Now fix $x,y\in H$.  For every continuous affine minorant
$\ell(z)=\langle p,z\rangle+b$ of $g_N$, the preceding argument gives
$\|p\|\leq L$, and therefore
\[
    \ell(x)
=
    \ell(y)+\langle p,x-y\rangle
\leq
    \ell(y)+L\|x-y\|
\leq
    f_N(y)+L\|x-y\|
.
\]
Taking the supremum over all continuous affine minorants $\ell\leq g_N$ gives
$
    f_N(x)\leq f_N(y)+L\|x-y\|
$.
Exchanging $x$ and $y$ yields
$
    |f_N(x)-f_N(y)|\leq L\|x-y\|
$,
and hence $\operatorname{Lip}(f_N)\leq L$.

For the last ``formula'' claim, since $\Delta_N$ is compact and the objective
in~\eqref{eq:R_Ndefo} is continuous, the infimum in~\eqref{eq:R_Ndefo} is
attained.  Hence
\[
    R_N(x)
=
    \min_{\lambda\in\Delta_N}
    \Big\{
        \sum_{n=1}^N \lambda_n\rho(\xi_n)
        +
        L\left\|x-\sum_{n=1}^N\lambda_n\xi_n\right\|
    \Big\}.
\]
It remains to prove that $R_N$ is the closed convex envelope of $g_N$.  We already know that
$R_N$ is convex, continuous, and $R_N\leq g_N$; hence
$R_N\leq g_N^{\star\star}$.
Conversely, fix $x\in H$ and $\lambda\in\Delta_N$.  Set
$
    y_\lambda\eqdef \sum_{n=1}^N\lambda_n\xi_n
$.  For each $n=1,\dots,N$, define
$
    z_n\eqdef \xi_n+x-y_\lambda
$; then $
    \sum_{n=1}^N\lambda_n z_n
=
    \sum_{n=1}^N\lambda_n\xi_n+x-y_\lambda
=
    x$.  
Moreover,
$
    g_N(z_n)
\leq
    \rho(\xi_n)+L\|z_n-\xi_n\|
=
    \rho(\xi_n)+L\|x-y_\lambda\|
$.  
Therefore, by convexity of $g_N^{\star\star}$ and since $g_N^{\star\star}\leq g_N$,
\[
\begin{aligned}
    g_N^{\star\star}(x)
    &=
    g_N^{\star\star}\biggl(\sum_{n=1}^N\lambda_n z_n\biggr)  \\
    &\leq
    \sum_{n=1}^N\lambda_n g_N^{\star\star}(z_n)  \\
    &\leq
    \sum_{n=1}^N\lambda_n g_N(z_n)  \\
    &\leq
    \sum_{n=1}^N\lambda_n\rho(\xi_n)
    +
    L\left\|x-\sum_{n=1}^N\lambda_n\xi_n\right\|.
\end{aligned}
\]
Taking the infimum over $\lambda\in\Delta_N$ gives
$
    g_N^{\star\star}(x)\leq R_N(x)
$.
Thus $g_N^{\star\star}=R_N$, which proves~\eqref{eq:semi_closedform}.
\end{proof}

The only trouble with the formula in~\eqref{eq:semi_closedform} is that the \textit{pointwise} minimum on the ``final layer'' is not genuinely explicit; namely the following must be \textit{exactly} computed 
\begin{equation}
\label{eq:please_dualize_me}
    \min_{\lambda\in\Delta_N}
    \,
    \biggl\{
        \sum_{n=1}^N \lambda_n y_n
        +
        L\Big\|
            x-\sum_{n=1}^N\lambda_n\xi_n
        \Big\|
    \biggr\}.
\end{equation}
over pairs of paired datapoints $\{(\xi_n,y_n)\}_{n=1}^N$.  The idea, is instead to relax the problem by \textit{entropic} penalizing the ``objective function'' $F_N:H\times \Delta_N\to \mathbb{R}$ defined implicitly in~\eqref{eq:semi_closedform} for each $(x,\lambda)\in H\times \Delta_N$ by
\begin{equation}
\label{eq:semi_closedform__simplicitvalue}
    F_N(x,\lambda)
=
    \sum_{n=1}^N \lambda_n\rho(\xi_n)
        +
        L\Big\|
            x-\sum_{n=1}^N\lambda_n\xi_n
        \Big\|
.
\end{equation}
\subsubsection{Step 2 - Discretization While Preserving Convexity and Lipschitzness}
% \subsection{Idea 2 - Dual discretization {\color{green}{no worko}}}
A na\"{i}ve approach is to discretize the $N$-simplex $\Delta_N$ which $F_N$ is optimized over (pointwise in $x$); however, since any such discretization is non-convex then the pointwise minimum can fail to be convex.  % Put counter-example in notes?
Instead, we will discretize $F_N$ by first considering its dual formulation of~\eqref{eq:please_dualize_me}.
\begin{lemma}[Dual form of $f_N$]
\label{lem:dualize}
In the setting of Lemma~\ref{lem:convexification_finite_whitney_mcshane}, for
every $x\in H$, $f_N$ is given by
\begin{equation}
\label{eq:I_dualized_you_bby}
    f_N(x)
=
    \sup_{\|p\|\leq L}
    \Big\{
        \langle p,x\rangle
        +
        \min_{1\leq n\leq N}
        \bigl(
            \rho(\xi_n)-\langle p,\xi_n\rangle
        \bigr)
    \Big\}.
\end{equation}
\end{lemma}
\begin{proof}[{Proof of Lemma~\ref{lem:dualize}}]
Using the Hilbert-space dual representation of the norm (by identifying $H\cong H^{\star}$ through the Riesz representation theorem) for each $z\in H$, we have that 
$
    L\|z\|
=
    \sup_{\|p\|\leq L}\langle p,z\rangle
$; whence we deduce that 
\[
\begin{aligned}
    f_N(x)
&=
    \min_{\lambda\in\Delta_N}
    \,
    \sup_{\|p\|\leq L}
    \,
    \Big\{
        \sum_{n=1}^N\lambda_n\rho(\xi_n)
        +
        \left\langle
            p,
            x-\sum_{n=1}^N\lambda_n\xi_n
        \right\rangle
    \Big\}  \\
&=
    \min_{\lambda\in\Delta_N}
    \sup_{\|p\|\leq L}
    \Big\{
        \langle p,x\rangle
        +
        \sum_{n=1}^N\lambda_n
        \bigl(
            \rho(\xi_n)-\langle p,\xi_n\rangle
        \bigr)
    \Big\}.
\end{aligned}
\]
For fixed $x$, define
$
    \Phi_x(\lambda,p)
\eqdef
    \langle p,x\rangle
    +
    \sum_{n=1}^N\lambda_n
    \bigl(
        \rho(\xi_n)-\langle p,\xi_n\rangle
    \bigr)
$.
The map $\Phi_x$ is affine in $\lambda$ and affine in $p$.  Moreover,
$\Delta_N$ is compact and convex, while $\overline B_H(0,L)$ is weakly compact
and convex, since $H$ is reflexive.  Since $p\mapsto\Phi_x(\lambda,p)$ is
weakly continuous for every $\lambda$, Sion's minimax theorem (cf.~\cite[Theorem 4.2']{sion1958general}) implies that
$
    \min_{\lambda\in\Delta_N}\sup_{\|p\|\leq L}\Phi_x(\lambda,p)
=
    \sup_{\|p\|\leq L}\min_{\lambda\in\Delta_N}\Phi_x(\lambda,p)
$.
Therefore, $f_N$ in~\eqref{eq:semi_closedform} can be re-expressed as 
\begin{equation}
\label{eq:almost_goot_iitbby}
    f_N(x)
=
    \sup_{\|p\|\leq L}
    \Big\{
        \langle p,x\rangle
        +
        \min_{\lambda\in\Delta_N}
        \sum_{n=1}^N\lambda_n
        \bigl(
            \rho(\xi_n)-\langle p,\xi_n\rangle
        \bigr)
    \Big\}
\end{equation}
Finally, since a linear functional over the simplex attains its minimum at an
extreme point,
\begin{equation}
\label{eq:sub_me_homie}
    \min_{\lambda\in\Delta_N}
    \sum_{n=1}^N\lambda_n
    \bigl(
        \rho(\xi_n)-\langle p,\xi_n\rangle
    \bigr)
=
    \min_{1\leq n\leq N}
    \bigl(
        \rho(\xi_n)-\langle p,\xi_n\rangle
    \bigr)
\end{equation}
substituting~\eqref{eq:sub_me_homie} into~\eqref{eq:almost_goot_iitbby}  yields~\eqref{eq:I_dualized_you_bby}.
\end{proof}
The advantage of the dual form is that one can discretize the set of $p\in H$ being optimized over in~\eqref{eq:I_dualized_you_bby}, namely the closed Hilbert ball $\overline{B}(0,L)\eqdef \{p\in H:\, \|p\|\le L\}$ of radius $L\ge 0$, into an $\eta$-net of a finite-dimensional sub-ball thereof; which is possible since finite dimensional balls are compact while infinite-dimensional ones are not (by the Hahn-Banach theorem).  
The reason why this discretization will not break convexity, unlike discretizing $\Delta_N$ directly in the primal form~\eqref{eq:semi_closedform}, is that~\eqref{eq:I_dualized_you_bby} expresses $f_N$ as the supremum of affine hyperplanes; and the point-wise \textit{any} supremum of (even possibly finitely) affine hyperplanes is still a convex function.

\begin{lemma}[Discretized dual form of $f_N$]
\label{lem:dualize_n_discretized}
In the setting of Lemma~\ref{lem:convexification_finite_whitney_mcshane}, for every non-empty compact subset $K\subseteq H$ and every $\varepsilon>0$, there exist $d\in \mathbb{N}_+$, a $d$-dimensional linear subspace $V_d\subseteq H$, $\eta>0$, and a finite $\eta$-net $p^\eta\eqdef \{p_m\}_{m=1}^M$ of $V_d\cap \overline{B}_H(0,L)$ such that the map $f_{N}^{p^{\eta}:V_d}:H\to \mathbb{R}$ defined by
\begin{equation}
\label{eq:I_dualized_discretized}
    f_{N}^{p^{\eta}:V_d}(x)
    \eqdef
    \max_{m=1,\dots,M}
    \,
    \Big\{
        \langle p_m,x\rangle
        +
        \min_{1\leq n\leq N}
        \bigl(
            \rho(\xi_n)-\langle p_m,\xi_n\rangle
        \bigr)
    \Big\},
    \qquad x\in H,
\end{equation}
satisfies
\begin{equation}
\label{eq:das_uniformli_estimatoski}
    \sup_{x\in K}
        \bigl|
            f_N(x)
            -
            f_N^{p^{\eta}:V_d}(x)
        \bigr|
    <
        \varepsilon .
\end{equation}
Moreover, $f_N^{p^{\eta}:V_d}$ is $L$-Lipschitz and convex on all of $H$.
\end{lemma}
\begin{proof}[{Proof of Lemma~\ref{lem:dualize_n_discretized}}]
For each $x\in H$ and $p\in \overline{B}_H(0,L)$, set
\[
    \Phi_x(p)
    \eqdef
    \langle p,x\rangle
    +
    \min_{1\leq n\leq N}
    \bigl(
        \rho(\xi_n)-\langle p,\xi_n\rangle
    \bigr)
    =
    \min_{1\leq n\leq N}
    \bigl(
        \rho(\xi_n)+\langle p,x-\xi_n\rangle
    \bigr).
\]
By Lemma~\ref{lem:dualize}, $f_N(x)=\sup_{\|p\|\le L}\Phi_x(p)$.

Choose $\alpha>0$ such that $L\alpha<\varepsilon/2$.  Since $K$ is compact, there exist $z_1,\dots,z_J\in K$ such that
$
    K\subseteq \bigcup_{j=1}^J B_H(z_j,\alpha)
$.
Set
$
    V_d\eqdef \operatorname{span}\{z_1,\dots,z_J,\xi_1,\dots,\xi_N\}
$,
and let $d\eqdef \dim(V_d)$.  Let $P_d:H\to V_d$ denote the orthogonal projection.  Since $\xi_n\in V_d$ for each $n=1,\dots,N$, and since $\sup_{x\in K}\operatorname{dist}(x,V_d)\leq \alpha$, we have
$
    \alpha_d\eqdef \sup_{x\in K}\max_{1\leq n\leq N}\|(I-P_d)(x-\xi_n)\|
    \leq
    \alpha
$.
For $p\in\overline{B}_H(0,L)$, set $q\eqdef P_d p$. Then $q\in V_d\cap \overline{B}_H(0,L)$ and, for every $x\in K$,
\[
    |\Phi_x(p)-\Phi_x(q)|
    \le
    \max_{1\leq n\leq N}
    |\langle p-q,x-\xi_n\rangle|
    =
    \max_{1\leq n\leq N}
    |\langle p,(I-P_d)(x-\xi_n)\rangle|
    \le
    L\alpha_d
    \leq
    L\alpha
    <
    \varepsilon/2.
\]
Defining $f_N^{V_d}(x)\eqdef \sup_{p\in V_d\cap \overline{B}_H(0,L)}\Phi_x(p)$, we obtain $0\le f_N(x)-f_N^{V_d}(x)<\varepsilon/2$ for every $x\in K$; where the lower bound holds since $V_d\cap\overline{B}_H(0,L)\subseteq \overline{B}_H(0,L)$ implies $f_N^{V_d}(x)\leq f_N(x)$.
It remains to replace $V_d\cap \overline{B}_H(0,L)$ by a finite net to deduce the approximation claim. Since $V_d$ is finite-dimensional, $V_d\cap \overline{B}_H(0,L)$ is compact. Let $R\eqdef \sup_{x\in K}\max_{1\leq n\leq N}\|x-\xi_n\|<\infty$. Choose $\eta>0$ such that $\eta R<\varepsilon/2$, with arbitrary $\eta>0$ if $R=0$. Let $p^\eta=\{p_m\}_{m=1}^M$ be an $\eta$-net of $V_d\cap \overline{B}_H(0,L)$. Then, for every $p\in V_d\cap \overline{B}_H(0,L)$, there exists $m\in\{1,\dots,M\}$ such that $\|p-p_m\|\le \eta$. Hence, for every $x\in K$ we have that
$
    |\Phi_x(p)-\Phi_x(p_m)|
    \le
    \max_{1\leq n\leq N}
    |\langle p-p_m,x-\xi_n\rangle|
    \le
    \eta R
$.
Taking suprema over $p\in V_d\cap \overline{B}_H(0,L)$ gives, for every $x\in K$,
\begin{equation}
\label{eq:estim_dos}
    0
    \le
    f_N^{V_d}(x)-f_N^{p^\eta:V_d}(x)
    \le
    \eta R
    <
    \varepsilon/2
.
\end{equation}
Combining the preceding finite-dimensional approximation estimate with~\eqref{eq:estim_dos} implies that, for every $x\in K$ we have
$
    0
\le
    f_N(x)-f_N^{p^\eta:V_d}(x)
\le
    f_N(x)-f_N^{V_d}(x)
    +
    f_N^{V_d}(x)-f_N^{p^\eta:V_d}(x)
<
    \varepsilon
$; implying~\eqref{eq:das_uniformli_estimatoski}.  
\hfill\\
\noindent
It remains to prove convexity and Lipschitzness.  For each $m\in\{1,\dots,M\}$, set $a_m\eqdef \min_{1\leq n\leq N}\bigl(\rho(\xi_n)-\langle p_m,\xi_n\rangle\bigr)$. Then $f_N^{p^\eta:V_d}(x)=\max_{1\leq m\leq M}\{\langle p_m,x\rangle+a_m\}$. For any $x,y\in H$ and $\lambda\in[0,1]$, let $m_\lambda\in\{1,\dots,M\}$ be a maximizer such that
$
    f_N^{p^\eta:V_d}(\lambda x+(1-\lambda)y)
=
    \langle p_{m_\lambda},\lambda x+(1-\lambda)y\rangle+a_{m_\lambda}
$.  
Then, by affinity of $H\ni z\mapsto \langle p_{m_\lambda},z\rangle+a_{m_\lambda}\in\mathbb{R}$ we have that
\[
\begin{aligned}
    f_N^{p^\eta:V_d}(\lambda x+(1-\lambda)y)
    &=
    \lambda\bigl(\langle p_{m_\lambda},x\rangle+a_{m_\lambda}\bigr)
    +
    (1-\lambda)\bigl(\langle p_{m_\lambda},y\rangle+a_{m_\lambda}\bigr) \\
    &\le
    \lambda f_N^{p^\eta:V_d}(x)
    +
    (1-\lambda)f_N^{p^\eta:V_d}(y).
\end{aligned}
\]
Thus $f_N^{p^\eta:V_d}$ is convex on $H$.
Moreover, since $p_m\in V_d\cap \overline{B}_H(0,L)$, we have $\|p_m\|\le L$ for every $m$. Thus, for every $x,y\in H$, we have that
\[
    f_N^{p^\eta:V_d}(x)-f_N^{p^\eta:V_d}(y)
    \le
    \max_{1\leq m\leq M}
    \bigl(
        \langle p_m,x\rangle+a_m
        -
        \langle p_m,y\rangle-a_m
    \bigr)
    \le
    \max_{1\leq m\leq M}
    |\langle p_m,x-y\rangle|
    \le
    L\|x-y\|.
\]
Interchanging $x$ and $y$ gives $f_N^{p^\eta:V_d}(y)-f_N^{p^\eta:V_d}(x)\le L\|x-y\|$. Hence $|f_N^{p^\eta:V_d}(x)-f_N^{p^\eta:V_d}(y)|\le L\|x-y\|$, and so $f_N^{p^\eta:V_d}$ is $L$-Lipschitz on $H$.
\end{proof}
It remains to relate our construction back to the original function being reconstructed from sample data; which we now demonstrate.
\begin{proof}[{Proof of Theorem~\ref{thrm:reconstruction} (Pt.\ 1)}]
Let $f_N$ be the convexified Whitney-McShane extension from Lemma~\ref{lem:convexification_finite_whitney_mcshane}.  By that lemma, $f_N$ interpolates the data, i.e. $f_N(\xi_n)=\rho(\xi_n)$ for every $n=1,\dots,N$.
Applying Lemma~\ref{lem:dualize_n_discretized} to the compact set $K\eqdef \mathbb{X}_N$, implies that: there exist $d\in\mathbb{N}_+$, $\eta>0$, and a finite $\eta$-net $p^\eta=\{p_m\}_{m=1}^M$ of $V_d\cap \overline{B}(0,L)$ such that
\begin{equation}
\label{eq:little_estimate}
    \sup_{x\in \mathbb{X}_N}
        \bigl|
            f_N(x)-f_N^{p^\eta:V_d}(x)
        \bigr|
    <
        \varepsilon .
\end{equation}
Since $f_N(\xi_n)=\rho(\xi_n)$, \eqref{eq:little_estimate} implies
\begin{equation}
\label{eq:das_uniformli_estimatoski_samples___unoski}
    \max_{n=1,\dots,N}
        \bigl|
            \rho(\xi_n)-f_N^{p^\eta:V_d}(\xi_n)
        \bigr|
=
    \max_{n=1,\dots,N}
        \bigl|
            f_N(\xi_n)-f_N^{p^\eta:V_d}(\xi_n)
        \bigr|
<
        \varepsilon .
\end{equation}
Now fix $\xi\in C$.  Since $\{\xi_n\}_{n=1}^N$ is a $\delta$-net of $C$ then, there exists $n_\xi\in\{1,\dots,N\}$ such that $\|\xi-\xi_{n_\xi}\|\le \delta$.  Using the $L$-Lipschitzness of $\rho$ and of $f_N^{p^\eta:V_d}$, together with \eqref{eq:das_uniformli_estimatoski_samples___unoski}, we obtain
\begin{equation}
\label{eq:riddle_me_dis}
\begin{aligned}
        \bigl|
            \rho(\xi)
            -
            f_N^{p^{\eta}:V_d}(\xi)
        \bigr|
&\le
        \bigl|
            \rho(\xi)
            -
            \rho(\xi_{n_\xi})
        \bigr|
    +
        \bigl|
            \rho(\xi_{n_\xi})
            -
            f_N^{p^{\eta}:V_d}(\xi_{n_\xi})
        \bigr|
    +
        \bigl|
            f_N^{p^{\eta}:V_d}(\xi_{n_\xi})
            -
            f_N^{p^{\eta}:V_d}(\xi)
        \bigr|
\\
&<
        L\|\xi-\xi_{n_\xi}\|
        +
        \varepsilon
        +
        L\|\xi-\xi_{n_\xi}\|
\\
&\le
        2L\delta+\varepsilon .
\end{aligned}
\end{equation}
Taking the supremum over $\xi\in C$ gives \eqref{eq:das_uniformli_estimatoski_samples}.  Finally, Lemma~\ref{lem:dualize_n_discretized} also implies that $f_N^{p^\eta:V_d}$ is convex and $L$-Lipschitz on all of $H$.
Without loss of generality, replace $\varepsilon$ by $\varepsilon/2$ and set $\delta\eqdef \varepsilon/(4L)$ throughout, so that $2L\delta+\varepsilon/2=\varepsilon$.
\end{proof}

\subsubsection{Step 3 - Estimating The Complexity of the Formula}
It remains to estimate the complexity of our reconstruction formula; i.e.\ the quantities $d$ and $M$ (as $N$ was already computed).  For this, we only require the following lemma.
\begin{lemma}[Finiteness and estimates for the finite-dimensional dual discretization]
\label{lem:finite_d_M_estimates}
Let $H$ be a separable Hilbert space, let $C\subseteq H$ be non-empty and compact, and let
$D\eqdef \operatorname{diam}(C)$.  For $\alpha>0$, define the effective linear dimension of $C$ at scale $\alpha$ by
\[
    d_C(\alpha)
    \eqdef
    \min
    \biggl\{
        \dim(V):
        V\subseteq H \text{ is a finite-dimensional linear subspace and }
        \sup_{\xi\in C}\operatorname{dist}(\xi,V)\leq \alpha
    \biggr\}.
\]
Then $d_C(\alpha)\le 
    N_{\alpha}(C)<\infty$.
\hfill\\
Let $L>0$, let $\mathbb{X}_N=\{\xi_n\}_{n=1}^N\subseteq C$, and fix a finite-dimensional subspace
$V_d\subseteq H$ satisfying
\[
    \sup_{\xi\in C}\operatorname{dist}(\xi,V_d)\leq \alpha.
\]
Then, for every $\eta>0$, the ball $V_d\cap \overline B_H(0,L)$ admits a finite $\eta$-net
$p^\eta=\{p_m\}_{m=1}^M$ satisfying
\[
    M
    \leq
    \biggl(1+\frac{2L}{\eta}\biggr)^d.
\]
\end{lemma}
\begin{proof}[{Proof of Lemma~\ref{lem:finite_d_M_estimates}}]
If $d=0$, then $V_d=\{0\}$ and $V_d\cap\overline B_H(0,L)=\{0\}$.  Thus one may take
$M=1$, and the bound
$
    M\leq (1+2L/\eta)^d
$
holds since both sides equal $1$.  We may therefore assume that $d\geq1$.
We first prove that $d_C(\alpha)$ is finite.  Since $C$ is compact, it is totally bounded.  Hence, for every $\alpha>0$, there exist points
$c_1,\dots,c_J\in C$ such that
$
    C
    \subseteq
    \bigcup_{j=1}^J B(c_j,\alpha)
$ and $
    J=N_\alpha(C)
$.  
Set $
    V\eqdef \operatorname{span}\{c_1,\dots,c_J\}
$.   Then $\dim(V)\leq J$, and for every $\xi\in C$ there exists $j\in\{1,\dots,J\}$ such that
$\|\xi-c_j\|\leq \alpha$.  Since $c_j\in V$, we obtain
$
    \operatorname{dist}(\xi,V)\leq \|\xi-c_j\|\leq \alpha
$.  
Taking the supremum over $\xi\in C$ gives
$
    \sup_{\xi\in C}\operatorname{dist}(\xi,V)\leq \alpha
$.  
Therefore $d_C(\alpha)\leq \dim(V)\leq N_\alpha(C)<\infty$.
\hfill\\
It remains to prove the estimate on $M$.  Since $V_d$ is $d$-dimensional Hilbert, it is isometric to $\mathbb{R}^d$ with its Euclidean norm.  Let $p^\eta=\{p_m\}_{m=1}^M$ be a maximal $\eta$-separated subset of
$V_d\cap \overline B_H(0,L)$.  By maximality, $p^\eta$ is an $\eta$-net of
$V_d\cap \overline B_H(0,L)$.  Since the balls
$\{
    B_{V_d}\bigl(p_m,\eta/2\bigr)
\}_{m=1}^M$ are pairwise disjoint, and since they are all contained in
$B_{V_d}(0,L+\eta/2)$, then, upon comparing $d$-dimensional Euclidean volumes gives
\[
    M\,\operatorname{vol}_d\bigl(B_{V_d}(0,\eta/2)\bigr)
\leq
    \operatorname{vol}_d\bigl(B_{V_d}(0,L+\eta/2)\bigr).
\]
Since Euclidean ball volumes scale like $r^d$, this yields
$
    M
\leq
    (
        \frac{L+\eta/2}{\eta/2}
    )^d
=
    (1+\frac{2L}{\eta})^d
$.
\end{proof}
\subsubsection{Completing the Proof}
We may now complete the proof of our main approximate-reconstruction formula.
\begin{proof}[{Proof of Theorem~\ref{thrm:reconstruction} (Continued)}]
It remains to justify that the parameters $d$ and $M$ can be chosen with the stated bounds.  Set
$
    \alpha\eqdef \varepsilon/(16L)
$
and choose
$
    d\eqdef d_C(\alpha)
$.
By Lemma~\ref{lem:finite_d_M_estimates}, there exists a $d$-dimensional linear
subspace $V_d\subseteq H$ satisfying
\[
    \sup_{\xi\in C}\operatorname{dist}(\xi,V_d)\leq \alpha,
\]
and
$
    d
=
    d_C\biggl(\frac{\varepsilon}{16L}\biggr)
\leq
    N_{\varepsilon/(16L)}(C)
<
    \infty
$.  
Next set
$
    \eta\eqdef \varepsilon/(8D)
$.
Again by Lemma~\ref{lem:finite_d_M_estimates}, the ball
$V_d\cap \overline B_H(0,L)$ admits an $\eta$-net
$p^\eta=\{p_m\}_{m=1}^M$ satisfying
$
    M
\leq
    \biggl(1+\frac{2L}{\eta}\biggr)^d
=
    \biggl(
        1+\frac{16LD}{\varepsilon}
    \biggr)^d
$.
We now verify that this choice of $V_d$ and $p^\eta$ gives the required reconstruction error.  For $x\in H$ and $p\in\overline B_H(0,L)$, set
$
    \Phi_x(p)
\eqdef
    \langle p,x\rangle
    +
    \min_{1\leq n\leq N}
    \bigl(
        \rho(\xi_n)-\langle p,\xi_n\rangle
    \bigr)
$.  
By Lemma~\ref{lem:dualize},
$
    f_N(x)=\sup_{\|p\|\leq L}\Phi_x(p)
$.  
Let $P_d:H\to V_d$ denote the orthogonal projection.  For every $x\in C$ and
$p\in\overline B_H(0,L)$, we have $P_dp\in V_d\cap\overline B_H(0,L)$ and $|\Phi_x(p)-\Phi_x(P_dp)|$ can be controlled from above by $2L\alpha$ since
\[
\begin{aligned}
    |\Phi_x(p)-\Phi_x(P_dp)|
    &\leq
    \max_{1\leq n\leq N}
    \bigl|
        \langle p-P_dp,x-\xi_n\rangle
    \bigr|  \\
    &=
    \max_{1\leq n\leq N}
    \bigl|
        \langle p-P_dp,(I-P_d)(x-\xi_n)\rangle
    \bigr|  \\
    &\leq
    L
    \max_{1\leq n\leq N}
    \|(I-P_d)(x-\xi_n)\|  \\
    &\leq
    L
    \max_{1\leq n\leq N}
    \bigl(
        \operatorname{dist}(x,V_d)
        +
        \operatorname{dist}(\xi_n,V_d)
    \bigr)  \\
    &\leq
    2L\alpha.
\end{aligned}
\]
Consequently, if
$
    f_N^{V_d}(x)
\eqdef
    \sup_{p\in V_d\cap\overline B_H(0,L)}\Phi_x(p)
$, 
then
$
    0
    \leq
    f_N(x)-f_N^{V_d}(x)
    \leq
    2L\alpha
$ for each $x\in C$.
Next, since $p^\eta$ is an $\eta$-net of $V_d\cap\overline B_H(0,L)$, for every
$p\in V_d\cap\overline B_H(0,L)$ there exists $m\in\{1,\dots,M\}$ such that
$\|p-p_m\|\leq \eta$.  Since $x,\xi_n\in C$, we have $\|x-\xi_n\|\leq D$, and hence
\[
\begin{aligned}
    |\Phi_x(p)-\Phi_x(p_m)|
\leq
    \max_{1\leq n\leq N}
    \bigl|
        \langle p-p_m,x-\xi_n\rangle
    \bigr| 
\leq
    \eta D.
\end{aligned}
\]
Thus, $
    0
\leq
    f_N^{V_d}(x)-f_N^{p^\eta:V_d}(x)
\leq
    \eta D
$ for each $x\in C$.  
Consequently, we deduce that
\[
    0
\leq
    f_N(x)-f_N^{p^\eta:V_d}(x)
\leq
    2L\alpha+\eta D
=
    \frac{\varepsilon}{8}
    +
    \frac{\varepsilon}{8}
=
    \frac{\varepsilon}{4}
.
\]
Therefore, 
$
    \sup_{\xi\in C}
    |f_N(\xi)-f_N^{p^\eta:V_d}(\xi)|
\leq
    \frac{\varepsilon}{4}
$.  
\hfill\\
\noindent
Finally, since $\mathbb{X}_N$ is a $\delta$-net of $C$, for each $\xi\in C$ there exists
$n_\xi\in\{1,\dots,N\}$ such that $\|\xi-\xi_{n_\xi}\|\leq\delta$.  Since $f_N$ interpolates the data and both $\rho$ and $f_N$ are $L$-Lipschitz,
\[
\begin{aligned}
    |\rho(\xi)-f_N(\xi)|
\leq
    |\rho(\xi)-\rho(\xi_{n_\xi})|
    +
    |\rho(\xi_{n_\xi})-f_N(\xi_{n_\xi})|
    +
    |f_N(\xi_{n_\xi})-f_N(\xi)|  
\leq
    2L\delta
    =
    \frac{\varepsilon}{2}.
\end{aligned}
\]
Hence, we compute that
\[
\begin{aligned}
    \sup_{\xi\in C}
    |\rho(\xi)-f_N^{p^\eta:V_d}(\xi)|
\leq
    \sup_{\xi\in C}|\rho(\xi)-f_N(\xi)|
    +
    \sup_{\xi\in C}|f_N(\xi)-f_N^{p^\eta:V_d}(\xi)|  
\leq
    \frac{\varepsilon}{2}
    +
    \frac{\varepsilon}{4}
    <
    \varepsilon.
\end{aligned}
\]
Lastly, it is enough to set
$
    N=N_{\delta}(C)=N_{\varepsilon/(4L)}(C),
$ and the bounds on $d$ and $M$ were proved above.  Since $f_N^{p^\eta:V_d}$ is a finite maximum of affine maps whose slopes have norm at most $L$, it is convex and $L$-Lipschitz on $H$.
\end{proof}

\subsection{{Proof of Neural Network Computation Result (Theorem~\ref{thrm:relu_MLP_computation})}}
\label{s:RepresentationGuarantee}
We now prove Theorem~\ref{thrm:relu_MLP_computation}.
\begin{proof}[{Proof of Theorem~\ref{thrm:relu_MLP_computation}}]
Let $f_N$ be the convexified Whitney-McShane extension from
Lemma~\ref{lem:convexification_finite_whitney_mcshane}.  Then $f_N$ interpolates
the data: 
$
    f_N(\xi_n)=\rho(\xi_n)=y_n
$ for each $n=1,\dots,N$.  
Apply Lemma~\ref{lem:dualize_n_discretized} to the compact set
$K\eqdef\{\xi_1,\dots,\xi_N\}$ with accuracy $\varepsilon/2$.  Then there exist
$M\in\mathbb{N}_+$ and directions $\{p_m\}_{m=1}^M\subseteq \overline B_H(0,L)$
such that the functional
\[
    H\ni 
    x\mapsto
    \max_{1\leq m\leq M}
    \Bigg\{
        \langle p_m,x\rangle
        +
        \min_{1\leq n\leq N}
        \bigl(
            y_n-\langle p_m,\xi_n\rangle
        \bigr)
    \Bigg\}
\]
satisfies
$    \max_{1\leq n\leq N}
    \bigl|
        f_N(\xi_n)-f_{\varepsilon,d}(\xi_n)
    \bigr|
<
    \frac{\varepsilon}{2}
$.  
Since $f_N(\xi_n)=\rho(\xi_n)$, this gives
$
    \max_{1\leq n\leq N}
    \,
    \bigl|
        \rho(\xi_n)-f_{\varepsilon,d}(\xi_n)
    \bigr|
    <
    \frac{\varepsilon}{2}
$.  
By~\cite[Lemma~5.11]{petersen2024mathematical}, for every $q\geq2$ there exists
a $\operatorname{ReLU}$-MLP realizing the minimum of $q$ inputs with size at most
$16q$, width at most $3q$, and depth at most $\lceil\log_2(q)\rceil$.  The same
lemma gives a $\operatorname{ReLU}$-MLP realizing the maximum with the same
complexity bounds.  Applying this with $q=N$ gives $\Psi_N$, and applying it
with $q=M$ gives $\Phi_M$.  If $N=1$ or $M=1$, the corresponding map is the
identity and the claim is trivial.
Therefore, for every $m=1,\dots,M$, we know that
\[
    \Psi_N
    \Big(
        \bigoplus_{n=1}^N
        \bigl(
            y_n-\langle p_m,\xi_n\rangle
        \bigr)
    \Big)
    =
    \min_{1\leq n\leq N}
    \bigl(
        y_n-\langle p_m,\xi_n\rangle
    \bigr),
\]
and hence, for every $x\in H$,
\[
\begin{aligned}
\Phi_M
    \Big(
        \bigoplus_{m=1}^M
        \left[
            \langle p_m,x\rangle
            +
            \Psi_N
            \Big(
                \bigoplus_{n=1}^N
                \bigl(
                    y_n-\langle p_m,\xi_n\rangle
                \bigr)
            \Big)
        \right]
    \Big)
=
    \max_{1\leq m\leq M}
    \Bigg\{
        \langle p_m,x\rangle
        +
        \min_{1\leq n\leq N}
        \bigl(
            y_n-\langle p_m,\xi_n\rangle
        \bigr)
    \Bigg\}.
\end{aligned}
\]
This proves the asserted neural-network representation.
% \hfill\\
% \noindent
Finally, assume that $\{\xi_n\}_{n=1}^N$ contains an $\varepsilon/(4L)$-net of
$C$.  Fix $\xi\in C$ and choose $n_\xi\in\{1,\dots,N\}$ such that
$\|\xi-\xi_{n_\xi}\|\leq \varepsilon/(4L)$.  Since both $\rho$ and
$f_{\varepsilon,d}$ are $L$-Lipschitz, we have
\[
\begin{aligned}
    |\rho(\xi)-f_{\varepsilon,d}(\xi)|
    &\leq
    |\rho(\xi)-\rho(\xi_{n_\xi})|
    +
    |\rho(\xi_{n_\xi})-f_{\varepsilon,d}(\xi_{n_\xi})|
    +
    |f_{\varepsilon,d}(\xi_{n_\xi})-f_{\varepsilon,d}(\xi)|  \\
    &<
    L\frac{\varepsilon}{4L}
    +
    \frac{\varepsilon}{2}
    +
    L\frac{\varepsilon}{4L}
    =
    \varepsilon.
\end{aligned}
\]
Taking the supremum over $\xi\in C$ proves the final claim.
\end{proof}

\subsection{{Proof of Convexity and Lipschitzness Certificate (Theorem~\ref{thrm:certificate})}}
\label{s:Proof__ss:certificate}
We use the following little lemma.
\begin{lemma}[Positive-weight PReLU-max networks are convex, monotone, and Lipschitz]
\label{lem:positive_weight_relu_max_convex_monotone}
Let $T:\mathbb{R}^{d_0}\to\mathbb{R}$ be defined by
\begin{align*}
    T(u)
&\eqdef
    {\bf A}^{(L-1)}z^{(L-1)}(u)+b^{(L-1)}
,
\\
    z^{(\ell+1)}(u)
&\eqdef
    \max_{\Pi^{(\ell)}}
    \operatorname{PReLU}^{\bullet}_{\mathbf{\alpha}^{(\ell)}}
    \bigl({\bf A}^{(\ell)}z^{(\ell)}(u)+b^{(\ell)}\bigr),
    \qquad
    \mbox{ for: }
    \qquad
    \ell=0,\dots,L-2
,
\\
    z^{(0)}(u)
& \eqdef 
    u
.
\end{align*}
Assume that ${\bf A}^{(\ell)}$ has nonnegative entries for every
$\ell=0,\dots,L-2$, that $\mathbf{\alpha}^{(\ell)}\in[0,1]^{d_{\ell+1}^{\prime}}$ for every
$\ell=0,\dots,L-2$, and that ${\bf A}^{(L-1)}$ is a nonnegative row vector.
Then $T$ is convex and coordinate-wise non-decreasing.  Moreover,
$
    \operatorname{Lip}(T)
\le
    \|{\bf A}^{(L-1)}\|_{2\to 2}
    \,
    \big(
        \prod_{\ell=0}^{L-2}
        \|{\bf A}^{(\ell)}\|_{2\to 2}
    \big)
$.
\end{lemma}

We are now ready to prove Theorem~\ref{thrm:certificate}.
\begin{proof}[{Proof of Theorem~\ref{thrm:certificate}}]
The input layer of the CNF may be written as
$
    x^{(0)}
=
    P(x)+q
$,
where $q\eqdef(q_1,\dots,q_M)\in\mathbb{R}^M$ and $P:H\to\mathbb{R}^M$ is the bounded linear map
$
    P(x)
\eqdef
    \bigoplus_{m=1}^M
    \langle p_m,x\rangle_H
$.
Let $T:\mathbb{R}^M\to\mathbb{R}$ denote the finite-dimensional PReLU-max network obtained from the hidden and output layers of Definition~\ref{defn:CNFs}.  Then
$
    \Phi(x)
=
    T(P(x)+q)
$.
Now, by Lemma~\ref{lem:positive_weight_relu_max_convex_monotone}, $T$ is convex.  Since $x\mapsto P(x)+q$ is affine, convexity is preserved under this pre-composition.  Hence, for every $x,y\in H$ and $t\in[0,1]$,
\[
\begin{aligned}
    \Phi(tx+(1-t)y)
    &=
    T(P(tx+(1-t)y)+q) \\
    &=
    T(t(P(x)+q)+(1-t)(P(y)+q)) \\
    &\le
    tT(P(x)+q)+(1-t)T(P(y)+q) \\
    &=
    t\Phi(x)+(1-t)\Phi(y).
\end{aligned}
\]
Thus $\Phi$ is convex.
To prove Lipschitzness, note that by Lemma~\ref{lem:positive_weight_relu_max_convex_monotone},
$
    \operatorname{Lip}(T)
\le
    \|{\bf A}^{(L-1)}\|_{2\to 2}
    \Big(
        \prod_{\ell=0}^{L-2}
        \|{\bf A}^{(\ell)}\|_{2\to 2}
    \Big)
$.
Therefore, for every $x,y\in H$,
\[
\begin{aligned}
    |\Phi(x)-\Phi(y)|
    &=
    |T(P(x)+q)-T(P(y)+q)| \\
    &\le
    \operatorname{Lip}(T)\|P(x)-P(y)\|_2 \\
    &\le
    \operatorname{Lip}(T)\|P\|_{H\to \ell_2^M}\|x-y\|_H.
\end{aligned}
\]
Hence
$
    \operatorname{Lip}(\Phi)
\le
    \|{\bf A}^{(L-1)}\|_{2\to 2}
    \Big(
        \prod_{\ell=0}^{L-2}
        \|{\bf A}^{(\ell)}\|_{2\to 2}
    \Big)
    \|P\|_{H\to \ell_2^M}
$.
Finally, by the Cauchy-Schwarz inequality, for every $x\in H$ we have
$
    \|P(x)\|_2^2
=
    \sum_{m=1}^M
    |\langle p_m,x\rangle_H|^2 
\le
    \Big(
        \sum_{m=1}^M
        \|p_m\|_H^2
    \Big)
    \|x\|_H^2
$.  
Thus
$
    \|P\|_{H\to \ell_2^M}
\le
    \Big(
        \sum_{m=1}^M
        \|p_m\|_H^2
    \Big)^{1/2}
$, 
which gives the following explicit Lipschitz bound
\[
    \operatorname{Lip}(\Phi)
\le
    \|{\bf A}^{(L-1)}\|_{2\to 2}
    \Big(
        \prod_{\ell=0}^{L-2}
        \|{\bf A}^{(\ell)}\|_{2\to 2}
    \Big)
    \Big(
        \sum_{m=1}^M
        \|p_m\|_H^2
    \Big)^{1/2}.
\]
\hfill\\
\noindent
It remains to prove the exact computability claim.  Let $f_{\varepsilon,d}$ be as in Theorem~\ref{thrm:relu_MLP_computation}, and set
$
    V\eqdef \operatorname{span}\{p_1,\dots,p_M\}
$.
For each $m=1,\dots,M$, define
$
    q_m
\eqdef
    \Psi_N
    \Big(
        \bigoplus_{n=1}^N
        \bigl(
            y_n-\langle p_m,\xi_n\rangle_H
        \bigr)
    \Big)
$.
Then the CNF input layer is exactly
$
    x^{(0)}_m
=
    \langle p_m,x\rangle_H+q_m
$.
Since $\operatorname{PReLU}_{1}$ is the identity map, the max layer appearing in~\eqref{eq:relu_representability_formula} is obtained by taking
$
    L=2
$, $
    d_0=M
$, $
    d_1=1
$, $
    d_1^{\prime}=M
$, $
    {\bf A}^{(0)}=I_M
$, $
    b^{(0)}=0
$, $
    \mathbf{\alpha}^{(0)}=\mathbf{1}_M
$, and $
    \Pi^{(0)}=\{[M]_+\}
$.
Then
$
    x^{(1)}
=
    \max_{m\in[M]_+} x^{(0)}_m
$.
Taking finally $
    {\bf A}^{(1)}=(1)
$ and $
    b^{(1)}=0
$ gives exactly $f_{\varepsilon,d}$.
This proves the claim.
\end{proof}

\bibliographystyle{acm}
\bibliography{Refs}
\end{document}